\documentclass[reqno,11pt]{amsart}
\usepackage{a4wide,amssymb,color,hyperref}
\usepackage{enumitem}
\setlist[enumerate]{label=\textnormal{(\roman*)}}
\usepackage{tikz}
\usetikzlibrary{decorations.pathreplacing}
 
\newcommand{\NN}{\mathbb{N}}
\newcommand{\RR}{\mathbb{R}}
\newcommand{\cA}{\mathcal{A}}

\newcommand{\cC}{\mathcal{C}}

\newcommand{\cE}{\mathcal{E}}
\newcommand{\cF}{\mathcal{F}}
\newcommand{\cI}{\mathcal{I}}
\newcommand{\cL}{\mathcal{L}}
\newcommand{\cN}{\mathcal{N}}
\newcommand{\cO}{\mathcal{O}}
\newcommand{\cR}{\mathcal{R}}
\newcommand{\cY}{\mathcal{Y}}
\newcommand{\cZ}{\mathcal{Z}}
\newcommand{\loc}{\textnormal{sol}}
\newcommand{\ONE}{\mathbf{1}}
\newcommand{\lesssimD}{\lesssim_\delta}

\DeclareMathOperator{\sech}{sech}
\DeclareMathOperator{\sign}{sign}
\DeclareMathOperator{\SPAN}{span}

\newtheorem{theorem}{Theorem}
\newtheorem{lemma}{Lemma}[section]
\newtheorem{proposition}[lemma]{Proposition}
\theoremstyle{remark}
\newtheorem{remark}{Remark}[section]
\numberwithin{equation}{section}

\begin{document}

\title[Finite point blowup for critical gKdV]{Finite point blowup for the critical generalized Korteweg-de Vries equation}

\author[Y. Martel]{Yvan Martel}
\address{CMLS, \'Ecole Polytechnique, CNRS, Institut Polytechnique de Paris, 91128 Palaiseau, France}
\email{yvan.martel@polytechnique.edu}

\author[D. Pilod]{Didier Pilod}
\address{Department of Mathematics, University of Bergen, Postbox 7800, 5020 Bergen, Norway }
\email{Didier.Pilod@uib.no}

\begin{abstract}
In the last twenty years, there have been significant advances in the study of the blow-up phenomenon for the critical generalized Korteweg-de~Vries equation, including the determination of sufficient conditions for blowup,
the stability of blowup in a refined topology and the classification of minimal mass blowup.
Exotic blow-up solutions with a continuum of blow-up rates and multi-point blow-up solutions were also constructed. 
However, all these results, as well as numerical simulations, involve the bubbling of a solitary wave going at infinity at the blow-up time, which means that the blow-up dynamics and the residue are eventually uncoupled. 
Even at the formal level, there was no indication whether blowup at a finite point could occur for this equation.
In this article, we answer this question by constructing solutions that blow up in finite time under the form of
a single-bubble concentrating the ground state at a finite point with an unforeseen blow-up rate.
Finding a blow-up rate intermediate between the self-similar rate
and other rates previously known also reopens the question of which blow-up rates are actually possible for this equation.
\end{abstract}

\maketitle

\section{Introduction}

\subsection{Problem setting}
The objective of this article is to shed some new light on the blow-up problem for the critical generalized Korteweg-de Vries equation
(gKdV)
\begin{equation} \label{gkdv}
\partial_t u+\partial_x\big(\partial_x^2u+u^5\big)=0,\quad (t,x) \in \RR\times \RR.
\end{equation}
Note that if $u$ is a solution of~\eqref{gkdv} then for any $(\lambda,\sigma) \in(0,+\infty) \times \RR$, the function
\begin{equation} \label{scaling:inv}
v(t,x):=\pm \frac 1{\lambda^{\frac12}} u\left(\frac t{\lambda^3}, \frac{x-\sigma}\lambda \right)
\end{equation}
and the function $w(t,x)=u(-t,-x)$ are also solutions of~\eqref{gkdv}.
Recall that the mass $M(u)$ and the energy $E(u)$ of a solution $u$ of \eqref{gkdv} are formally conserved
\begin{align*}
M(u)&=\int_{\RR} u^2 \, dx \\
E(u)&=\frac12 \int_{\RR} (\partial_xu)^2 \, dx-\frac16 \int_{\RR} u^6 \, dx .
\end{align*}
The Cauchy problem for~\eqref{gkdv} is locally well-posed in the energy space $H^1(\RR)$ by~\cite{KPV,KPV2}.
For any $u_0 \in H^1(\RR)$, there exists a unique (in a certain class) maximal solution $u$ of~\eqref{gkdv} in $\cC([0,T), H^1(\RR))$ satisfying $u(0,\cdot)=u_0$.
Moreover, by \cite[Corollary 1.4]{KPV2}, if the maximal time of existence $T$ is finite then 
\[
\liminf_{t \uparrow T} (T-t)^{\frac 13}\|\partial_x u(t)\|_{L^2}>0.
\]

\smallskip

We define the function
\[
Q(x)=\left(3\sech^2(2x)\right)^{1/4}
\]
unique even positive $H^1$ solution of the equation
\begin{equation} \label{eq:Q}
-Q''+Q-Q^5=0 \quad \mbox{on $\RR$} .
\end{equation}
Recall the sharp Gagliardo-Nirenberg inequality~\cite{Wei}
\begin{equation} \label{sharpGN}
\frac 13 \int_{\RR} \phi^6 \le \left(\frac{\int_{\RR} \phi^2}{\int_{\RR} Q^2} \right)^2 \int_{\RR} (\partial_x\phi)^2, \quad \forall \, \phi \in H^1(\RR) 
\end{equation}
and the variational characterization of $Q$: for any $\phi\in H^1$,
\begin{equation}\label{var:char}
0< \|\phi\|_{L^2} \leq \|Q\|_{L^2} \mbox{ and } E(\phi)\leq 0 \implies \exists \lambda>0, \sigma\in \RR : 
\phi(x)= \pm\frac 1{\lambda^{\frac12}}Q\left(\frac {x - \sigma}{\lambda} \right).
\end{equation}
By~\eqref{sharpGN} and the conservation of mass and energy, for
any initial data $u_0 \in H^1(\RR)$ satisfying $\|u_0\|_{L^2} < \|Q\|_{L^2}$, the corresponding solution of~\eqref{gkdv} is global and bounded in $H^1(\RR)$.

\smallskip

For any $(\lambda,\sigma) \in(0,+\infty) \times \RR$, the function
\begin{equation}\label{def:sol} 
u(t,x)=\frac 1{\lambda^{\frac12}}Q\left(\frac 1{\lambda} \left( x - \frac t{\lambda^{2}} - \sigma\right)\right)
\end{equation}
is a solution of~\eqref{gkdv} called solitary wave or soliton. We observe the relation between the scale~$\frac 1\lambda$ of the soliton and
its speed of propagation $\frac 1{\lambda^2}$. 

\subsection{First bubbling results for gKdV}

For many nonlinear partial differential equations, like the semilinear wave equation, the semilinear heat equation and the 
nonlinear Schr\"odinger equation, blowup was proved for large classes of initial data by a contradiction argument involving a rather simple functional (like the virial identity for the nonlinear Schr\"odinger equation). Using this strategy, blowup is a consequence of an obstruction to global existence and little information on the blow-up phenomenon is obtained.
In the absence of such simple criterion of non global existence for the critical gKdV,
advances on the blow-up problem could only be made by addressing the blow-up phenomenon through \emph{bubbling} (or concentration) of the ground state. A first hint in favor of this strategy is given by explicit blow-up solutions 
for the mass critical nonlinear Schr\"odinger equation \eqref{def:SNLS}.
Another hint is suggested by \cite[Theorem 4]{We86} which shows using~\eqref{var:char} that
a blow-up solution of \eqref{gkdv} with the threshold mass $\|u_0\|_{L^2}=\|Q\|_{L^2}$ necessarily has the soliton behavior at the blow-up time
$T$
\[
\lambda^{\frac 12}(t) u (t, \lambda(t) x + \sigma(t)) \to \pm Q \quad \mbox{as $t\uparrow T$ in $H^1$ strong}
\]
for some functions $\sigma(t)\in \RR$ and $\lambda(t)\downarrow 0$ as $t\uparrow T$.
To go beyond the threshold mass $\|Q\|_{L^2}$, it is natural to restrict to 
solutions with mass slightly above the threshold, \emph{i.e.} satisfying
\begin{equation}\label{mass:close}
\|Q\|_{L^2} < \|u_0\|_{L^2} < (1+\delta) \|Q\|_{L^2}
\end{equation}
for $0<\delta\ll 1$.
Using a rigidity property of the gKdV flow around the family of solitons established in \cite{MaMejmpa}, the first proof of blowup in finite or infinite time was given in \cite{Mjams} for any negative energy initial data satisfying \eqref{mass:close}.
Then, a direct link was established between the blow-up phenomenon and the solitary wave in \cite[Theorem 1]{MaMe}, which states that
any blow up solution $u$ of \eqref{gkdv} satisfying \eqref{mass:close} has the following bubbling behavior
\begin{equation}\label{bubbling}
\lambda^{\frac 12}(t) u (t, \lambda(t) x + \sigma(t)) \rightharpoonup \pm Q \quad \mbox{as $t\uparrow T$ in $H^1$ weak.}
\end{equation}
where $0<T\leq \infty$ is the (possibly infinite) blow-up time.
The above weak convergence shows that the soliton $Q$ is the universal profile for the blowup phenomenon, but also allows for various possibilities of
small residue and functions $\lambda$ and $\sigma$. At that stage, it was unknown whether blow up could occur in finite or infinite time and what was the asymptotic behavior of the parameters $\lambda$ and
$\sigma$ close to the blow-up time.
However, it was proved that
\begin{equation}\label{sig:lam}
\lim_{t\uparrow T} \lambda^2(t) \sigma'(t) =1 .
\end{equation}
These general results later appeared to be optimal. Not only
blowup can occur either in finite or infinite time, justifying the alternative in~\cite{Mjams}, but also various blow-up rates are possible both for finite time and infinite time blowup, justifying the flexibility on $\lambda$ and $\sigma$ in~\eqref{bubbling}.

\subsection{Stable blowup and minimal mass blowup for gKdV}
First, it is proved in \cite{MaMeRa1} that
there exists a subset of initial data, included and open for $\|\cdot\|_{H^1}$ in the set
\[
\cA = \left\{ u_0 = Q+\varepsilon_0 : \varepsilon_0 \in H^1,\ \|\varepsilon_0\|_{H^1} <\delta_0 \mbox{ and } \int_{x>0} x^{10} \varepsilon_0^2 dx <1\right\}
\]
where $\delta_0>0$ is small,
leading to blow-up solutions of \eqref{gkdv} such that
\[
\lim_{t\uparrow T} \left(T-t\right) \|\partial_x u(t)\|_{L^2} = \frac {\|Q'\|_{L^2}}{\ell}
\]
for some $\ell=\ell(u_0)>0$.
Because of the openness property, we call this phenomenon \emph{stable blowup}. 
The asymptotic behavior of these solutions at the 
blow-up time $T$ is typical of the bubbling phenomenon. Indeed, there exists a function $r^\star\in H^1$ such that
\begin{equation}\label{uSS}
\lim_{t\uparrow T} \left\{ u(t,\cdot) - \frac 1{\lambda^{\frac 12}(t)}
Q\left(\frac{\cdot-\sigma(t)}{\lambda(t)} \right)\right\} = r^\star \quad \mbox{in $L^2$ strong}
\end{equation}
for some functions $\lambda$, $\sigma$ satisfying 
\[
\lim_{t\uparrow T} \frac{\lambda(t)}{T-t} = \ell
\quad \mbox{and} \quad \lim_{t\uparrow T} \left(T-t\right)\sigma(t) = \frac 1{\ell^2}\,.
\]

Second, \cite[Theorem 1.3]{MaMeRa2} settles the case of blowup at the threshold mass $\|u_0\|_{L^2}=\|Q\|_{L^2}$
showing the existence of a unique (up to invariances) blow-up solution $S(t)$ of
\eqref{gkdv} with the mass of $Q$, blowing up in finite time with the rate $\|S(t)\|_{H^1} \sim C(T-t)^{-1}$ and concentrating a soliton at the position $\sigma(t) \sim C (T-t)^{-1}$ as $t\uparrow T$. 
A precise description of the time and space asymptotic behavior of $S$ in~\cite{CoMa1}
has permitted to construct multi-point bubbling in~\cite{CoMa2}.
Such multi-bubble solutions provide examples of blowup with an arbitrary number of blow-up bubbles,
all propagating simultaneously to~$\infty$ as $t\uparrow T$.

\subsection{Exotic blowup for gKdV}

Because of assumption $u_0\in \cA$,
the above stable blow-up theory does not say the whole story in the $H^1$ neighborhood of the family of solitons.
From~\cite{MaMeRa3}, there exists a large class of \emph{exotic} finite time blow-up solutions of \eqref{gkdv}, close to the family of solitons, with blow-up rates of the form $\|\partial_x u(t)\|_{L^2} \sim (T-t)^{-\nu}$ for any $\nu\geq \frac {11}{13}$.
Such blow-up rates are generated by the nonlinear interactions of the bubbling soliton with an explicit slowly decaying tail added to the initial data. Because of the tail, these $H^1$ solutions fail to belong to $\cA$ and the trichotomy of \cite{MaMeNaRa,MaMeRa1} does not apply.
The value $\frac {11}{13}$ is not optimal in~\cite{MaMeRa3}, but a non trivial extension of the proof would be needed to improve it significantly, for example to approach the special value $\nu = \frac 12$.
For a one-bubble blow-up solution, by the relation \eqref{sig:lam} between the position $\sigma$ of the bubble and its blow up rate $\lambda$,
the rate $\nu = \frac 12$ is critical:
if $\nu\geq \frac 12$ then the soliton goes at $\infty$ since $\lim_{T}\sigma= \infty$, while for $\frac 13 < \nu < \frac 12$, the soliton has to stay in a compact interval.
The existence of bubbling solutions with $\nu= \frac 13$ was ruled out in~\cite{MaMe}, but
the range $\frac 13 < \nu < \frac 12$ was left open.

\subsection{Supercritical case and numerical experiments}
Numerical experiments on blowup for \eqref{gkdv} and its supercritical counterpart
$\partial_t u + \partial_x (\partial_x^2 u + u^p)=0$ for $p>5$, are reported in several articles.
 Pioneering numerical studies \cite{BDKM,DM} and more recent works~\cite{KP} exhibited self-similar blow-up for the supercritical equation
 which was confirmed rigorously by~\cite{Koch,Lansuper} for $p>5$ close to $5$.
Self-similar blowup means $\|\partial_x u(t)\|_{L^2}\sim (T-t)^{-\frac 13}$ close to the blow-up time and in particular blow-up at a finite point.
This is the only known blowup rate for the supercritical case, in contrast with the critical case where it was ruled out in~\cite{MaMe}.
The article~\cite{Amodio} is especially devoted to the supercritical case in the limit $p\downarrow 5$.
In the critical case, the article~\cite{KP} corroborates the stable blowup described in~\cite{MaMeRa1}, even for large initial data.

\subsection{Main result}

As a conclusion of the above review, we observe that despite many efforts spent on the blowup phenomenon for critical gKdV, the literature does not give any evidence on whether blowup at a finite point could occur, even at the formal level.
The main result of this article answers this question by proving the existence of a family of bubbling solutions of~\eqref{gkdv} with the key feature that the blow-up point converges to a finite value at the blow-up time, as a consequence of an unforeseen blow-up rate.

\begin{theorem} \label{th:1}
For any $\delta>0$, there exist $T>0$ and a solution $u\in \cC([0,T),H^1(\RR))$ of~\eqref{gkdv} of the form
\[
u(t,x) = \frac 1{(T-t)^{\frac 15}} Q\left(\frac {x-\sigma(t)}{(T-t)^{\frac 25}}\right) + r(t,x),
\]
where $\lim_{t\uparrow T} \sigma(t) =0$
and the function $r$ satisfies
\begin{equation*}
\sup_{t\in [0,T)}\|r(t)\|_{L^2}\leq \delta\quad \mbox{and}\quad
\lim_{t\uparrow T} \left(T-t\right)^{\frac 25} \|\partial_x r(t)\|_{L^2} =0.
\end{equation*}
In particular, $u$ blows up at time $T$ and
\[
\lim_{t\uparrow T} \left(T-t\right)^{\frac 25} \|\partial_x u(t)\|_{L^2} = \|Q'\|_{L^2}.
\]
\end{theorem}

We refer to Section~\ref{S:6} for more detailed asymptotics, including
$\lim_{t\uparrow T} (T-t)^{-\frac15} \sigma(t) =-5$.
\begin{remark}\label{rk:1.1}
From Theorem~\ref{th:1} and a rescaling argument, we remark that finite point blowup occurs for some initial data in $H^1$ arbitrarily close to $Q$ in $H^1$. 
Let $\delta>0$ and a solution $u$ as in Theorem~\ref{th:1} with blow-up time $T>0$.
Let $t_0\in [0,T)$ and $\lambda_0=(T-t_0)^{\frac 25}$. By the invariances, the function
$\widetilde u (t,x) = \lambda_0^\frac 12 u\left(\lambda_0^3 t + t_0 , \lambda_0 x + \sigma(t_0) \right)$
is still a solution of \eqref{gkdv} on $[0,\widetilde T)$ where $\widetilde T = \frac {T-t_0}{\lambda_0^3}=(T-t_0)^{-\frac 15}$
and it blows up at time $\widetilde T$ at a finite point.
Moreover, by Theorem~\ref{th:1},
\[
\widetilde u(t,x)= \mu^\frac 12(t) Q\left( \mu(t) x - y(t)\right)+ \lambda_0^\frac 12 r\left(\lambda_0^3 t + t_0 , \lambda_0 x + \sigma(t_0) \right)
\]
where
\[
\mu(t) = \frac{\lambda_0}{(T-\lambda_0^3 t - t_0)^{\frac 25}} \quad\mbox{and}\quad
y(t) =\frac{\sigma(\lambda_0^3 t + t_0)-\sigma(t_0)}{(T-\lambda_0^3 t - t_0)^\frac25}.
\]
Thus, at $t=0$, the solution $\widetilde u$ satisfies
\[
\| \widetilde u (0) - Q \|_{L^2} = \|r(t_0)\|_{L^2} \leq \delta,\quad
\| \partial_x \widetilde u (0) - Q' \|_{L^2} = (T-t_0)^{\frac 25} \|\partial_x r(t_0)\|_{L^2}.
\]
Taking $\delta>0$ small and $t_0\in [0,T)$ close enough to $T$, $\widetilde u(0)$ is arbitrarily close to $Q$ in $H^1$.
It follows from the classification in~\cite{MaMeRa1} that $\widetilde u(0)\not\in \cA$ (see also Section~\ref{S:Sketch}).
\end{remark}

Figure 1 recapitulates known blow-up rates for \eqref{gkdv}, \emph{i.e.} values $\nu>0$ for which there
exist $H^1$ blow-up solutions $u$ of \eqref{gkdv} with
$\lim_{t\uparrow T} (T-t)^\nu \|\partial_x u \|_{L^2} = C$ for $C>0$.

\smallskip

\hskip -0.5cm
\begin{tikzpicture}[scale=1]
\draw[arrows=->,line width=.4pt,thick](-4,1.5)--(-4,1);
\draw[arrows=->,line width=.4pt,thick](2.1,0.6)--(1.6,0.1);
\draw (4.04,0.85) node{{\footnotesize Stable blowup in $\cA$ \cite{MaMeRa1}}};
\draw (4.25,0.4) node{{\footnotesize Minimal mass blow up \cite{MaMeRa2}}};
\draw [dashed] (-7.8,0) -- (-5.3,0);
\draw (-5.13,0) -- (0.3,0);
\draw [very thick] (0.3,0) -- (6.5,0);
\draw [dashed, very thick] (6.5,0) -- (7.5,0);
\draw (-5.2,0) circle (2pt);
\filldraw (1.5,0) circle (2pt);
\filldraw (-4,0) circle (2pt);
\draw (-2.5,0.1) -- (-2.5,-0.1);
\draw [very thick] (0.3,0.1) -- (0.3,-0.1);
\draw (-7.8,0.1) -- (-7.8,-0.1);
\draw (-7.8,0.6) node{{\footnotesize $0$}};
\draw (-2.5,0.6) node{{\footnotesize $\frac 12$}};
\draw (-4,0.6) node{{\footnotesize $\frac 25$}};
\draw (0.3,0.6) node{{\footnotesize $\frac {11}{13}$}};
\draw (-5.2,0.6) node{{\footnotesize $\frac 13$}};
\draw (1.5,0.6) node{{\footnotesize $1$}};
\draw (7,0.6) node{{\footnotesize $\nu\to\infty$}};
\draw (-4,1.8) node{{\footnotesize Theorem \ref{th:1}}};
\draw [decorate,decoration={brace,amplitude=10pt}] (0.3,1.2) -- (7.5,1.2) node [black,midway,xshift=0.00cm,yshift=0.6cm]
{\footnotesize Existence proved in \cite{MaMeRa3}};
\draw [decorate,decoration={brace,amplitude=10pt}] (7.5,-0.2) -- (-2.5,-0.2) node [black,midway,xshift=0.00cm,yshift=-0.7cm]
{\footnotesize Blow-up point at $\infty$};
\draw [decorate,decoration={brace,amplitude=10pt}] (-2.5,-0.2) -- (-5.2,-0.2) node [black,midway,xshift=0.00cm,yshift=-0.7cm]
{\footnotesize Finite blow-up point};
\draw [decorate,decoration={brace,amplitude=10pt}] (-7.8,1.2) -- (-5.2,1.2) node [black,midway,xshift=0.00cm,yshift=0.6cm]
{\footnotesize Ruled out \cite{KPV2,MaMe}};
\draw [decorate,decoration={brace,amplitude=5pt}] (2.35,0.2) -- (2.35,1.1) node {\footnotesize};
\end{tikzpicture}

\begin{center}
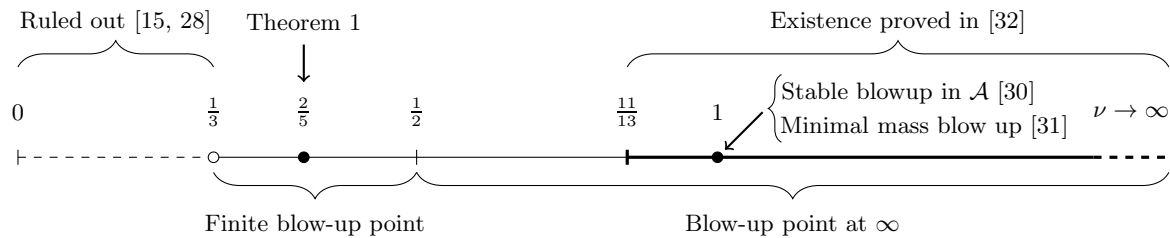
Figure 1. Blow-up rates for the critical gKdV equation
\end{center}

Now, we discuss values $\nu>\frac 13$ that were not considered in previous works.
First, extending the method in \cite{MaMeRa3} should provide the whole range $\nu\in [\frac 12,\infty)$.
Second, by constructing a solution with the blow-up rate $\nu=\frac 25$, the present paper reopens the question of which blow-up rates are possible for the equation, even close to the self-similar rate.
We conjecture that all blow-up rates are possible in the range $\nu\in (\frac 13,\frac 12)$.
Since the value $\nu=\frac 25$ corresponds to a rather exceptional structure, described in Section~\ref{S:Sketch}, we expect that some key additional ingredients will be needed to obtain the range 
$\nu\in (\frac 13,\frac 12)$ or a substantial part of it.
Non power-like blow-up rates may also be possible.

\subsection{Comparison with NLS}

We briefly discuss the blow-up problem for the $L^2$ critical nonlinear Schr\"odinger equation (NLS), 
restricting ourselves to the one-dimensional case
\begin{equation}\label{nls}
i \partial_t u + \partial_x^2 u + |u|^4 u = 0,\quad (t,x)\in \RR\times \RR.
\end{equation}
We refer to \cite{Cabook} for general information concerning \eqref{nls}.
The first notable fact concerning blowup is the existence of the explicit blow-up solution
\begin{equation}\label{def:SNLS}
S_\textnormal{NLS}(t,x)=\frac 1{T-t} e^{-\frac{i|x|^2}{4(T-t)} +\frac{i}{T-t}} Q\left( \frac{x}{T-t}\right)
\end{equation}
which has minimal mass for blowup, \emph{i.e.} $\|S_\textnormal{NLS}(0)\|_{L^2}=\|Q\|_{L^2}$, and satisfies the blow-up behavior
$\|\partial_x S_\textnormal{NLS}(t)\|_{L^2} \sim C(T-t)^{-1}$, often called \emph{conformal blowup}.
Moreover, it was proved in~\cite{Me93} that $S_\textnormal{NLS}$ is the unique (up to the invariances of \eqref{nls}) blow-up solution with the minimal mass.
Then, other constructions of solutions with the blow-up rate $(T-t)^{-1}$ appeared in \cite{BW} (see also \cite{KS2,MRS}
for instability properties).
In a different direction, the study of the so-called $\log\log$ blowup for \eqref{nls} has a long history, where formal arguments and numerical experiments long prevailed, see \cite{SulemSulem}.
Rigorous proofs first appeared in \cite{MeRa04,MeRa05,MeRa05bis,MeRa06,MeRa07,P}. In particular, it is known that there exists an open set in $H^1$ of initial data (including negative energy functions close to the soliton) leading to blowup in finite time with the $\log \log$ speed
\[
\| \partial_x u (t) \|_{L^2} \sim C \sqrt{\frac {\log|\log (T-t)|}{T-t}} \quad \mbox{as $t\uparrow T$}
\]
where $C$ is a universal constant.
The $\log\log$ speed is a slight perturbation of the self-similar blowup $(T-t)^{-\frac 12}$ for NLS.
Moreover, it follows from \cite{Ra} that a finite time blow-up solution of \eqref{nls} with \eqref{mass:close} either enjoys the $\log \log$ blowup, or blows up with a rate $\|\partial_x u(t)\|_{L^2} \geq C (T-t)^{-1}$, which represents an important gap in possible blowups.
Last, we point out that there is only one example of blowup different from the $\log \log$ and the conformal blowups, of the form $\|\partial_x u(t)\|_{L^2} \sim C|\log (T-t)| (T-t)^{-1}$ but involving at least two interacting bubbles, constructed in~\cite{MaRa}.
Figure 2 recapitulates this information.

\smallskip

\begin{tikzpicture}[scale=1]
\draw[arrows=->,line width=.4pt,thick](-0.85,1.3)--(-0.85,0.3);
\draw[arrows=->,line width=.4pt,thick](4,1.3)--(4,0.8);
\draw [dashed] (-6,0) -- (-1,0);
\draw (-1,0) -- (7.5,0);
\draw (-6,0.1) -- (-6,-0.1);
\draw (-1,0.1) -- (-1,-0.1);
\draw [dashed] (7.5,0) -- (8.5,0);
\filldraw (-0.85,0) circle (2pt);
\filldraw (4,0) circle (2pt);
\draw (-6,0.6) node{{\footnotesize $0$}};
\draw (-1,0.6) node{{\footnotesize $\frac 12$}};
\draw (4,0.6) node{{\footnotesize $1$}};
\draw (4,1.5) node{{\footnotesize Conformal blowup \eqref{def:SNLS}, \cite{BW}}};
\draw (-0.85,1.5) node{{\footnotesize Log-log blowup \cite{MeRa05bis,MeRa06,P}}};
\draw (8,0.6) node{{\footnotesize $\nu\to\infty$}};

\draw [decorate,decoration={brace,amplitude=10pt}] (3.9,-0.2) -- (-0.8,-0.2) node [black,midway,xshift=0.00cm,yshift=-0.7cm]
{\footnotesize Ruled out near one bubble \cite{Ra}};
\draw [decorate,decoration={brace,amplitude=10pt}] (-0.9,-0.2) -- (-6,-0.2) node [black,midway,xshift=0.00cm,yshift=-0.7cm]
{\footnotesize Ruled out};
\end{tikzpicture}

\begin{center}
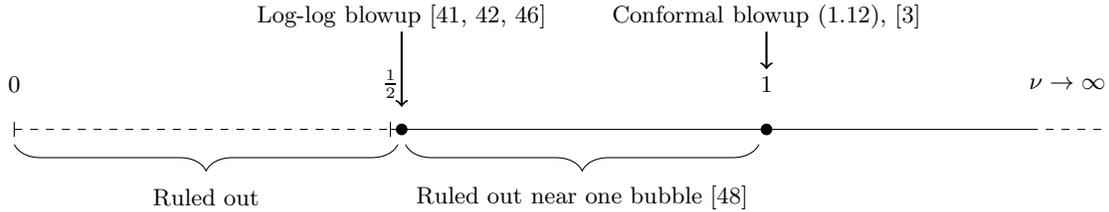
Figure 2. Blow-up rates for the mass critical NLS equation with one bubble
\end{center}

Several differences with the gKdV case are thus apparent. First, the existence of a continuum of blow-up rates
 has already been proved for gKdV but is open for NLS. Moreover, for critical gKdV, the existence of a gap property in blow-up rates as the one stated for NLS now seems unlikely. 
Last, the influence of the choice of a weighted topology is clearly established for gKdV but not for NLS.

\subsection{Related results for gKdV}
In the framework of stable blowup, 
as a consequence of~\cite{MaMeNaRa, MaMeRa1}, there exists a local $\cC^1$ co-dimension one manifold included in 
$\cA$ which separates the stable blow-up behavior from 
solutions that eventually exit the soliton neighborhood.
The solutions on the manifold are global in time and converge in a local norm to the family of solitary waves.
In the same context, the continuation of solutions after the blow-up time was studied in \cite{Lan}.
The exact large time behavior of solutions exiting any small soliton neighborhood (up to invariances)
is unknown even if scattering is expected. This question is related to the so-called \emph{scattering conjecture}, saying
 that any solution with $\|u_0\|_{L^2}< \|Q\|_{L^2}$ should actually scatter. Such questions were successfully addressed for the critical NLS equation but are still largely open for gKdV. In~\cite{KKSV,Ta}, this conjecture was first studied both for focusing and defocusing cases. In~\cite{Do}, this conjecture is fully solved for the defocusing critical gKdV.
For the fosusing case, it is proved in ~\cite{DG} that for any $H^1$ initial data with mass smaller than the mass of the soliton and close to the soliton in $L^2$ norm, the corresponding solution of \eqref{gkdv} must eventually exit any small soliton neighborhood. 
See also \cite[Theorem~1.3]{MuPo} and its proof for a related result.

\smallskip

Concerning blow-up properties, it is proved in~\cite{Pi} that blow-up solutions of \eqref{gkdv}
concentrate at least the mass of the ground state at the blow-up time.
The long time behavior of solutions close to the soliton familly is also studied for \emph{saturated nonlinearities} of the form $u^5 - \gamma |u|^{p-1} u$, for fixed $p>5$, in the limit $\gamma\downarrow 0$~in~\cite{Lanstar}.
In~\cite{MaPi}, a different type of solutions of \eqref{gkdv},
called \emph{flattening solitons}, are constructed : for any $\nu\in(0,\frac 13)$,
there exist global solutions, arbitrarily close to $Q$, such that as $t\to \infty$,
\[
u(t,x) = t^{-\frac \nu2}Q\left( t^{-\nu} (x-\sigma(t)) \right) + r(t,x)
\]
where $\sigma(t)\sim C t^{1-2\nu}$ and the function $r(t)$ is small in $L^2$ and converges to $0$ locally in the neighborhood of the soliton.

\smallskip

The literature concerning \emph{type II blowup} (in contrast with ODE blowup) for energy critical wave equations and parabolic problems is also vast and beyond the scope of this brief review.

\subsection{Sketch of the proof}\label{S:Sketch}
We construct the solution by compactness, passing to the weak limit in a sequence of solutions close to an appropriate ansatz and defined backwards in time on a uniform interval of time. We refer to~\cite{Me,RaSz} for pioneering works using this strategy in the blow-up context and to the subsequent articles~\cite{CoMa1,LeMaRa,MaPimBO}.
By this method, no stability property is obtained, even up to a finite number of instability directions, unlike in~\cite{MaMeRa1,MaMeRa3,MaPi}.

\smallskip

For the simplicity of notation, we first look for a solution $U$ of~\eqref{gkdv} blowing up at 
$(t,x)=(0,0)$.
As usual in investigating blow-up phenomenon, we introduce rescaled variables $(s,y)$, 
$s<0$ and $y\in \RR$, setting
\begin{equation} \label{u:w}
U(t,x) = \frac 1{\lambda^{\frac 12}(s)} w(s,y) ,\quad \frac {ds}{dt}= \frac 1{\lambda^3},\quad y = \frac {x-\sigma(s)}{\lambda(s)}.
\end{equation}
The time dependent parameters $\lambda>0$, $\sigma\in \RR$ are to be determined and the function $w$ satisfies the rescaled equation
\begin{equation} \label{rescaled}
\partial_s w + \partial_y ( \partial_y^2 w - w + w^5) 
- \frac{\lambda_s}{\lambda}\left( \frac w 2 + y \partial_y w \right)
- \left( \frac{\sigma_s}{\lambda} - 1 \right) \partial_y w =0.
\end{equation}
Setting $w=Q+v$ and using \eqref{eq:Q}, the small function $v$ satisfies
\begin{equation}\label{eq:v}
\partial_s v - \partial_y ( \cL v ) 
-\frac{\lambda_s}{\lambda}\Lambda Q
-\left( \frac{\sigma_s}{\lambda} - 1 \right) Q' =\mbox{nonlinear terms}
\end{equation}
where the operators $\Lambda$ and $\cL$ are defined in \eqref{def:lambda} and \eqref{def:L}.

\smallskip

To define a suitable approximate solution, we use an expansion $v=v_1+v_2+\cdots$ in powers of $\lambda^{\frac 12}$ 
up to a sufficiently high order. The main novelty lies on the definition of  $v_1$
\begin{equation}\label{def:v1}
v_1(s,y) = - 2 \lambda^{\frac 12}(s) \Theta \left( \lambda(s) y + \sigma(s)\right) A_1(y)
\end{equation}
where the function $A_1$ is odd, such that $(\cL A_1) ' = \Lambda Q$ and $\lim_{\pm\infty} A_1 = \mp m_0$
($m_0$ is a positive constant) and $\Theta$ is a smooth cut-off function equal to $1$ in a small neighborhood of $0$.
The introduction of such a function $A_1$ is related to the existence of a \emph{resonance} for the operator 
$\partial_x\cL$ of the form $1+R$ where $R$ is a Schwartz function and $\partial_x \cL (1+R)=0$.
Using a function~$A_1$ which is not localized for $y>0$, in contrast with the function $P$ used in \cite{MaMeRa1,MaMeRa2,MaMeRa3},
allows new blow-up constructions for critical gKdV. Moreover, such an ansatz justifies formally the claim that $\widetilde u(0)\not\in \cA$ in Remark~\ref{rk:1.1}. We refer to Lemma~\ref{le:2.1} for the definition of these functions.
Observe also that $\lambda^{\frac 12}(s) \Theta \left( \lambda(s) y + \sigma(s)\right)$ is the rescaled counterpart of $\Theta(x)$.

\smallskip

Now, we derive the blow-up rate obtained with this ansatz.
Using that $s$-differentiation as well as the $y$-differentiation of the cut-off function
$\Theta \left( \lambda(s) y + \sigma(s)\right)$ are negligible, we have
\[
\partial_s v_1 - \partial_y ( \cL v_1 ) \approx
- 2\lambda^{\frac 12}(s) \Theta \left( \lambda(s) y + \sigma(s)\right) \Lambda Q(y) .
\]
Inserting this in the equation \eqref{eq:v} of $v$, the following approximate blow-up law appears
\[
\frac{\lambda_s}{\lambda} \approx 2 \lambda^{\frac 12} .
\]
We identify a special solution $\lambda(s) = s^{-2}$ (recall that $s<0$).
Going back to the original time variable using $dt = \lambda^3(s) ds$, this gives 
blowup in finite time $t=0$ with
\begin{equation}\label{eq:1.16}
t \approx \frac 1{5|s|^{5}}, \quad \lambda(t) \approx (5t)^{\frac 25}, \quad
\sigma(t)\approx (5t)^{\frac 15}.
\end{equation}
For $T>0$ small, the solution $u$ of Theorem~\ref{th:1} is then given by $u(t,x) = \frac 1{5^{1/2}} U\left( \frac{T-t}{5^3} , -\frac{x}5\right)$.

\smallskip

Finally, we discuss more technical aspects of the proof.
First, in the expression \eqref{def:v1} of~$v_1$, both the oddness of the function $A_1$ and the power in the multiplicative factor $\lambda^{\frac 12}$ are decisive in vanishing error terms
(see the proof of Proposition~\ref{pr:error}).
We also point out that a key ingredient of the proof is a sharp control of the variation of the energy of the approximate solution based on a sufficiently small error term.
In Section \ref{Sec:5}, the difference between the solution and the ansatz is controlled uniformly
using a variant of the virial-energy functional first introduced in~\cite{MaMeRa1}.
Last, it is relevant to write the ansatz in the original variables. 
From~\eqref{u:w}, we have formally
\begin{equation*}
U(t,x) \approx \frac 1{\lambda^{\frac 12}(t)} Q \left(\frac{x-\sigma(t)}{\lambda(t)}\right) 
- 2 \Theta(x) A_1\left(\frac{x-\sigma(t)}{\lambda(t)}\right) .
\end{equation*}
By~\eqref{eq:1.16} and the asymptotic properties of $A_1$ in \eqref{def:A1}, the first correction term of the ansatz  behaves like $r_*(x)= 2 m_0 \Theta(x) \sign(x)$ at the blow-up time where $m_0>0$ is a constant~\eqref{def:m0}.
Such a discontinuous blow-up residual is in contrast with the $H^1$ regularity of the asymptotic residual corresponding to the stable blowup~\eqref{uSS}.
Closeness to a self-similar rate is generally associated to a residue with low regularity, see~\cite{MeRa07} for the mass critical NLS equation and \cite{HR,JJ,JL1,JL2,KS,KST1,KST2,Ro} for related observations in the energy critical wave-type models.

\subsection{Notation}
Throughout the paper, we will denote by $c$ a positive constant independent of $\delta$ which may change from line to line.
The notation $a \lesssim b$ means that $a \le c b$.
The notation $a \lesssimD b$ will be used for the estimate $a \le c_\delta b$, where the constant $c_\delta$ may depend on $\delta$.
\\
We use the notation $y\mapsto \lfloor y \rfloor$ for the floor function which maps $y$ to the greatest integer less than or equal to $y$.\\
For $1 \le p \le +\infty$, $L^p(\RR)$ denotes the standard Lebesgue spaces. 
For $f,g \in L^2(\RR)$ two real-valued functions, we denote the scalar product 
$(f,g)=\int_{\RR} f(x)g(x) dx$.
From now on, for simplicity of notation, we write $\int$ instead of $\int_{\RR}$ and we often omit~$dx$.\\
Let $\chi:\RR\to [0,1]$ be a $\cC^{\infty}(\RR)$ nondecreasing function such that
\begin{equation} \label{def:chi}
\chi_{|(-\infty,-2)} \equiv 0 \quad \text{and} \quad \chi_{|(-1,+\infty)}\equiv 1 .
\end{equation}
We define the weight function $\omega:\RR\to\RR$ by
\[
\omega(y) = e^{-\frac {|y|}2}
\]
and the weighted norm
\begin{equation} \label{def:L2sol}
\|f\|_{L^2_\loc}=\left(\int f^2(y) e^{-\frac{|y|}{10}} dy \right)^ {\frac12} .
\end{equation}
We introduce the generator of the scaling symmetry 
\begin{equation} \label{def:lambda}
\Lambda f=\frac12 f +yf' 
\end{equation}
and for any $k \in \NN$,
\begin{equation} \label{def:Lambda:k}
\Lambda_kf = \frac{1-k}2f +yf'.
\end{equation}
We also define the linearized operator $\cL$ around the ground state by 
\begin{equation} \label{def:L}
\cL f=-f''+f-5Q^4f .
\end{equation}
We define the universal constant 
\begin{equation} \label{def:m0} 
m_0:=\frac14 \int Q >0.
\end{equation}
Let $\cY$ be the set of functions $\phi:\RR\to\RR$ of class $\cC^{\infty}(\RR)$ such that
\begin{equation*}
\forall \, k \in \NN, \, \exists \, C_k > 0, \, r_k \ge 0 \text{ such that } |\phi^{(k)}(y)| \le C_k(1+|y|)^{r_k}e^{-|y|}, \ \forall \, y \in \RR .
\end{equation*}
For any $k\geq 0$, we fix a function $z_k:\RR\to[0,\infty)$ of class $\cC^\infty$ such that
$z_k(y)=|y|^{k}$ for $y \le -1$ and 
$z_k(y)=0$ for $y>0$ and  we define 
\begin{equation}\label{charac:Zk}
\cZ_k = \cY + \SPAN(z_0,\ldots,z_k).
\end{equation}

\section{Blow-up profile}

This section is devoted to the construction of a blow-up profile, \emph{i.e.} an approximate solution close to $Q$ of the rescaled equation~\eqref{rescaled} with the expected blow-up behavior with a sufficiently high order of precision.

\subsection{The linearized operator}
We recall standard properties of the operator $\cL$ and introduce useful functions for the construction of the blow-up profile.

\begin{lemma}[{\cite[Lemma 2.1]{MaMeRa1}}]\label{prop:L}
The self-adjoint operator $\cL$ on $L^2(\RR)$ defined by~\eqref{def:L} satisfies the following properties.
\begin{enumerate} 
\item \emph{Spectrum of $\cL$.} The operator $\cL$ has only one negative eigenvalue $-8$ associated to the eigenfunction $Q^3$,
$\ker \cL=\{aQ' : a \in \RR\}$ and $\sigma_\textnormal{ess}( \cL)=[1,+\infty)$.
\item \emph{Scaling.} $\cL\Lambda Q=-2Q$ and $(Q,\Lambda Q)=0$ where $\Lambda$ is defined in~\eqref{def:lambda}.
\item \emph{Coercivity of $\cL$.}
There exists $\nu_0>0$ such that, for all $\phi \in H^1(\RR)$,
\begin{equation} \label{coercivity.2}
(\cL\phi,\phi) \ge \nu_0\|\phi\|_{H^1}^2-\frac1{\nu_0}\left((\phi,Q)^2+(\phi,y\Lambda Q)^2 +(\phi,\Lambda Q)^2 \right) .
\end{equation}
\item \emph{Invertibility.} For any function $h \in \cY$ orthogonal to $Q'$ for $(\cdot,\cdot)$,
there exists a unique function $f \in \cY$ orthogonal to $Q'$ such that $\cL f=h$; moreover, if $h$ is even (resp. odd) then $f$ is even (resp. odd). 
\end{enumerate}
\end{lemma}

\begin{lemma} \label{le:2.1}
\begin{enumerate}
\item There exists a unique even function $R \in \cY$ such that 
$\cL R=5Q^4$.
\item There exists a unique function $P \in \cZ_0$ such that 
\begin{equation} \label{def:P}
(\cL P)'=\Lambda Q, \quad \lim_{y \to - \infty}P(y)= 2m_0, \quad \lim_{y \to +\infty} P(y)=0,\quad (P,Q')=0.
\end{equation}
Moreover,
\begin{equation} \label{P:Q} 
(P,Q)=m_0^2 .
\end{equation}
\item Let $A_1=P-m_0(1+R)$.
Then, $A_1$ is a bounded, odd function of class $\cC^\infty$ satisfying
\begin{equation} \label{def:A1}
A_1' \in \cY,\quad (\cL A_1)'=\Lambda Q \quad \mbox{and}\quad \lim_{y \to \mp \infty}A_1(y)=\pm m_0.
\end{equation}
Moreover,
\begin{equation} \label{A1:Q} 
(A_1,Q)=0 ,\quad (A_1,Q')=0.
\end{equation}
\end{enumerate}
\end{lemma}

\begin{proof} 
The existence of the function $R$ follows directly from (iv) of Lemma~\ref{prop:L},
and the function $P$ was first introduced in Proposition 2.2 of~\cite{MaMeRa1}.
The values of the limits $\lim_{\pm \infty} A_1$ are deduced from $R\in \cY$ and~\eqref{def:P}.
Moreover, it is clear from its definition that $A_1$ is bounded, of class $\cC^\infty$ and satisfies $A_1'\in \cY$.
By~$\cL R=5Q^4$, we have $(\cL (1+R))'=0$, which yields $(\cL A_1)'=\Lambda Q$ thanks to~\eqref{def:P}.
Now, decompose $A_1=A_e+A_o$, where $A_e$ is even and $A_o$ is odd. Since $\lim_{+\infty} A_1= - \lim_{-\infty} A_1$,
we have $A_e\in \cY$. Moreover, $(\cL A_e)'=0$ since $\Lambda Q$ is even. Thus, by (i) of Lemma~\ref{prop:L}, $A_e=0$. Last,~\eqref{A1:Q} follows by parity and~\eqref{def:P}.
\end{proof}

\subsection{Definition and properties of the blow-up profile}\label{S:2.2}
We fix for the rest of the paper an even smooth function $\Theta_0:\RR\to [0,1]$ with compact support such that
\begin{equation*}
\Theta_0(x)= \begin{cases}
1 & \mbox{if $|x|<\frac 54$,}\\
0 & \mbox{if $|x|>\frac 74$.}
\end{cases}
\end{equation*}
and we set $C_0=\|\Theta_0\|_{L^2}^2>0$.
Let $\delta\in (0,1)$ be small enough and define
\[
\Theta(x) = \Theta_0\left( \frac x \delta \right)
\]
so that
\begin{equation} \label{def:Theta}
\Theta(x)= \begin{cases}
 1 & \mbox{if $|x|<\frac 54\delta$,}\\
 0 & \mbox{if $|x|>\frac 74\delta$}
\end{cases}
\end{equation}
and
\begin{equation} \label{L2:Theta0}
\|\Theta\|_{L^2}^2 = C_0 \delta .
\end{equation}
Let $s_0<0$, $|s_0|\gg1$ to be fixed later and let $\cI\subset(-\infty,s_0]$ be a compact interval. 
We consider $\cC^1$ real-valued functions $\lambda:\cI\to (0,\infty)$ and $\sigma:\cI\to \RR$
such that
\begin{equation} \label{wBT}
0 < \lambda < \delta^4\quad \mbox{and} \quad - \frac {\delta^2} 4< \sigma < 0.
\end{equation}
Define
\begin{equation} \label{def:theta}
\theta(s,y) =\theta(y;\lambda(s),\sigma(s))= \lambda^{\frac 12}(s) \Theta(\lambda(s) y + \sigma(s)).
\end{equation}
For functions 
\begin{equation} \label{def:AAA}
A_2 \in \cY, \quad A_3, \, A_4, \, A_5^{\star} \in \cZ_1, \quad A_5 \in \cZ_2,
\end{equation}
and real constants $c_1$, $c_2$, $c_3$, $c_4^{\star}$ to be chosen later, we set
\begin{equation} \label{def:V_j}
\begin{aligned}
&V_1=c_1 A_1, \ V_2=c_2 P+A_2, \ V_3=c_3 P + A_3,\ V_4=A_4,\ V_5=A_5, \\
&V_4^{\star}=c_4^{\star}P, \ V_5^{\star}=A_5^{\star},
\end{aligned}
\end{equation}
and
\begin{equation} \label{def:V}
V(s,y)=\sum_{j=1}^5V_j(y) \theta^j(s,y)+(\log\lambda(s))\sum_{k=4}^5 V_k^{\star}(y) \theta^k(s,y).
\end{equation}
Recall that the functions $P$ and $A_1$ are defined in Lemma~\ref{le:2.1}.
Let the blow-up profile be
\begin{equation} \label{def:W}
 W(s,y)= W(y;\lambda(s),\sigma(s)) = Q(y)+V(s,y).
\end{equation}
Last, set
\begin{align} 
\beta(s,y)&= \sum_{j=1}^3 c_j \theta^j(s,y)+c_4^{\star}(\log\lambda(s)) \theta^4(s,y)\label{def:beta}\\
\widetilde{\beta}(s)&=\sum_{j=1}^3 c_j \lambda^{\frac{j}2}(s)+c_4^{\star} \left(\log\lambda(s)\right) \lambda^2(s) . \label{def:beta_tilde}
\end{align}
The form of the approximate solution $W$
is justified in the proof of the next proposition, where we fix the
functions $A_j$ and $A_k^\star$ and the constants $c_2$, $c_3$, $c_4^\star$.
\begin{lemma}\label{le:VVV}
If $A_2,A_3,A_4,A_5,A_5^{\star}$ satisfy~\eqref{def:AAA} then
\begin{equation}\label{VVV}\begin{aligned}
&V_1\in L^\infty, \quad V_1'\in \cY, \quad \Lambda_1 V_1\in \cY,\\
&V_2,V_4^\star \in \cZ_0, \quad V_3, \, V_4, \, V_5^{\star} \in \cZ_1, \quad V_5 \in \cZ_2,\\
&\Lambda_2 V_2, \, \Lambda_3 V_3, \, \Lambda_4 V_4^\star\in \cZ_0,\quad
\Lambda_4 V_4, \, \Lambda_5 V_5, \, \Lambda_5 V_5^\star \in \cZ_1.
\end{aligned}\end{equation}
For $j=2,3,4,5$ and $k=4,5$,
\eqref{VVV} is summarized by
\begin{equation}\label{VVVbis}
V_j\in \cZ_{\lfloor \frac{j-1}2\rfloor},\quad
\Lambda_j V_j\in \cZ_{\lfloor \frac{j-2}2\rfloor},\quad
V_k^\star,\Lambda_k V_k^\star\in \cZ_{\lfloor \frac{k-3}2\rfloor}.
\end{equation}
\end{lemma}
\begin{proof}
For $V_j$ and $V_k^\star$, \eqref{VVV} is a consequence of the definitions in \eqref{def:AAA}-\eqref{def:V_j} and the properties
of $P$ and $A_1$ in Lemma~\ref{le:2.1}.
Next, we observe the following cancellation: if $\phi \in \cZ_l$ for $l \geq 1$, then $\Lambda_{2l+1} \phi \in \cZ_{l-1}$.
In particular, $\Lambda_3 V_3\in \cZ_0$ and $\Lambda_5 V_5\in \cZ_1$.
\end{proof}
Let
\begin{equation} \label{eq:W.1}
\cE (W ) 
= \partial_s W + \partial_y ( \partial_y^2 W - W + W^5) 
- \frac{\lambda_s}{\lambda}\Lambda W - \left( \frac{\sigma_s}{\lambda} - 1 \right) \partial_y W.
\end{equation}

\begin{proposition} \label{pr:error}
Assume~\eqref{wBT}. Let $c_1=-2$. 
There exist real constants $c_2$, $c_3$, $c_4^{\star}$ and functions 
$A_2,A_3,A_4,A_5,A_5^{\star}$ satisfying~\eqref{def:AAA},
such that the function $W$ satisfies the following properties,
for all $s\in \cI$.
\begin{enumerate}
\item \emph{Pointwise estimates.}
For any $p\geq 0$, for any $y\in \RR$,
\begin{equation}
|\partial_y^p W| + |\partial_y^p \Lambda W|\lesssim \omega + \delta^{-p} \lambda^{p+\frac 12} \ONE_{[-2\delta,2\delta]}(\lambda y) . \label{est:W0}
\end{equation}
\item \emph{Error of $W$ for the rescaled equation.}
\begin{equation} \label{eq:W.2}
\cE(W)=-\left(\frac{\lambda_s}{\lambda}+\beta\right) \left( \Lambda Q+\Psi_{\lambda} \right)
- \left( \frac{\sigma_s}{\lambda} - 1 \right) (Q'+\Psi_{\sigma})+ \Psi_{W}
\end{equation}
where for any $p\geq 0$, for any $y\in \RR$,
\begin{align}
\left|\partial_y^p \Psi_{\lambda} \right| &\lesssimD \lambda^{\frac12}\omega+ \lambda^{1+p} \ONE_{[-2\delta,0]}(\lambda y) , \label{est:Psi_lambda}\\
\left|\partial_y^p \Psi_{\sigma} \right| &\lesssimD \lambda^{\frac12}\omega+ \lambda^{\frac32+p}\ONE_{[-2\delta,0]}(\lambda y) , \label{est:Psi_sigma}\\
\left|\partial_y^p \Psi_{W} \right| &\lesssimD \left|\log\lambda\right| \lambda^{3} \left( |y| \ONE_{[-2\delta,0]}(\lambda y) 
+ \omega \right)
+\lambda^{\frac72+p}\ONE_{[0,2\delta]}(\lambda y) .
\label{est:Psi_W}
\end{align}
Moreover, 
\begin{equation} \label{est:Psi_lambda:Q}
 \left|\left(\Psi_{\lambda},Q\right) \right| \lesssimD \lambda .
\end{equation}
\item \emph{Mass of $W$.}
\begin{equation} \label{est:asymp:mass:W}
\left| \int W^2 - \int Q^2 \right| \lesssim \delta.
\end{equation}
\item \emph{Variation of the energy of $W$.}
\begin{equation} \label{est:deriv:energy:W}
\left|\frac{d}{ds}\left[\frac {E(W)}{\lambda^2}\right]\right|
\lesssimD \frac 1{\lambda^2} \left( \lambda^{\frac12}\left| \frac{\lambda_s}{\lambda}+\widetilde{\beta} \right|
+ \lambda^{\frac12}\left| \frac{\sigma_s}{\lambda}-1\right|+ \left| \log\lambda \right| \lambda^3 \right).
\end{equation}
\end{enumerate}
\end{proposition}
We choose $c_1=-2$ to simplify the expression of the scaling asymptotic (see Section~\ref{S:2.3}).
Subsections~\ref{s.const}--\ref{s.variation} are devoted to the proof of Proposition~\ref{pr:error}, while
Subsection~\ref{S:2.3} gives the heuristic blow-up law.

\subsection{Construction of the blow-up profile}\label{s.const}
First, for any $j \in \NN$, we compute
\begin{equation}\label{eq:dsthj}
\partial_s (\theta^j) = j\frac{\lambda_s}{\lambda} \theta^{j-1}\Lambda \theta + j\frac{\sigma_s}{\lambda} \theta^{j-1}\partial_y \theta ,
\end{equation}
and thus
\begin{equation*}
\begin{aligned}
\partial_s W &= \frac{\lambda_s}{\lambda} \left(\sum_{j=1}^5 jV_j \theta^{j-1} \Lambda \theta 
+(\log\lambda) \sum_{k=4}^5kV_k^{\star} \theta^{k-1} \Lambda \theta+ \sum_{k=4}^5V_k^{\star}\theta^k\right)\\
& \quad + \frac{\sigma_s}{\lambda} \left(\sum_{j=1}^5 jV_j \theta^{j-1}\partial_y \theta
+(\log\lambda) \sum_{k=4}^5 kV_k^{\star} \theta^{k-1}\partial_y\theta\right) .
\end{aligned}
\end{equation*}
Second, using $W=Q+V$, $Q''-Q+Q^5=0$ and the definition of $\cL$,
\[
\partial_y \left( \partial_y^2 W - W + W^5\right)
=\partial_y\left( - \cL V + \left(Q+V\right)^5 - Q^5 - 5 Q^4 V\right).
\]
Thus, expanding $V$, we obtain
\begin{align*}
\partial_y \left( \partial_y^2 W - W + W^5\right)
&=-\sum_{j=1}^5(\cL V_j)' \theta^j-(\log \lambda)\sum_{k=4}^5(\cL V_k^{\star})'\theta^k \\ 
& \quad -\sum_{j=1}^5jV_j \theta^{j-1}\partial_y\theta-(\log \lambda)\sum_{k=4}^5 kV_k^{\star}\theta^{k-1}\partial_y\theta \\
& \quad +\partial_y \left((Q+V)^5-Q^5-5Q^4V\right)+\Psi_1
\end{align*}
where $\Psi_1$ is an error term defined by
\begin{equation} \label{def:Psi:theta1} 
\begin{aligned}
\Psi_1&:=
5 Q^4 \sum_{j=1}^5 V_j\partial_y(\theta^j) + 5(\log \lambda) Q^4 \sum_{k=4}^5 V_k^{\star}\partial_y(\theta^k)
\\& \quad + \sum_{j=1}^5 \left(3V_j'' \partial_y( \theta^j) + 3 V_j' \partial_y^2 (\theta^j)
+ V_j\partial_y^3( \theta^j)\right) \\
& \quad +(\log \lambda)\sum_{k=4}^5\left( 3(V_k^{\star})'' \partial_y( \theta^k)
+3 (V_k^{\star})' \partial_y^2( \theta^k) +V_k^{\star} \partial_y^3( \theta^k)\right).
\end{aligned}
\end{equation}
Third, we compute
\begin{equation*}
\partial_yW=Q'+\sum_{j=1}^5\left(jV_j\theta^{j-1}\partial_y\theta+ V_j'\theta^j \right)+
(\log\lambda)\sum_{k=4}^5 \left( kV_k^{\star}\theta^{k-1}\partial_y\theta+ (V_k^{\star})'\theta^k \right).
\end{equation*}
Last, we remark that
$\Lambda(\theta^j V_j) = jV_j \theta^{j-1}\Lambda \theta + (\Lambda_j V_j) \theta^j$,
where $\Lambda_j$ is defined in \eqref{def:Lambda:k}. Thus,
\begin{equation*}
\Lambda W=\Lambda Q+ \sum_{j=1}^5 \left(jV_j \theta^{j-1} \Lambda \theta +(\Lambda_j V_j) \theta^j\right)
+(\log\lambda)\sum_{k=4}^5 \left(k V_k^{\star} \theta^{k-1}\Lambda \theta + (\Lambda_k V_k^{\star})\theta^k \right).
\end{equation*}
Combining these identities, using the expressions of $V_j$ and $V_k^\star$ in \eqref{def:V_j}, and the
equations of $P$ and $A_1$ in~\eqref{def:P} and~\eqref{def:A1}, we deduce that 
\begin{align*}
\cE(W) &= - \left(\frac{\lambda_s}{\lambda}+\beta\right)\left(\Lambda Q +\Psi_\lambda\right)
- \left( \frac{\sigma_s}{\lambda} - 1 \right) \left( Q'+\Psi_{\sigma} \right)\\
& \quad - \sum_{j=2}^5(\cL A_j)'\theta^j - (\log\lambda) (\cL A_5^{\star})'\theta^5\\
& \quad +\partial_y \left((Q+V)^5-Q^5-5Q^4V\right)+\beta \Psi_\lambda +\Psi_1 
\end{align*}
where
\begin{align}
\Psi_{\lambda} &=\sum_{j=1}^5(\Lambda_jV_j)\theta^j-\sum_{k=4}^5V_k^{\star} \theta^k
+(\log\lambda)\sum_{k=4}^5 (\Lambda_kV_k^{\star}) \theta^k ,\label{def:Psi:lambda}\\
\Psi_{\sigma}&=\sum_{j=1}^5 V_j'\theta^j+(\log\lambda) \sum_{k=4}^5(V_k^{\star})'\theta^k .\label{def:Psi:sigma}
\end{align}
Denote $V_0:=Q$ for the simplicity of notation.
We expand
\begin{equation*}
(Q+V)^5-Q^5-5Q^4V = \sum_{j=2}^5 M_j \theta^j+ (\log\lambda) M_5^\star \theta^5 +\sum_{l=0}^5 (\log\lambda)^l \left( \sum_{j=6}^{25} M_{l,j} \theta^j \right)
\end{equation*}
where for $j=2,\ldots,5$,
\begin{align*}
M_j & = \left(\sum_{j_1+\cdots+j_5=j } V_{j_1} \cdots V_{j_5} \right)-5V_0^4 V_j, \\
M_5^{\star} & = 20 V_0^3 V_1V_4^{\star},
\end{align*}
and for $l=0,\ldots,5$, $j=6,\ldots,25$,
\[
M_{l,j}=\sum_{j_1+\cdots+ j_{5-l}+k_{1}+\cdots+k_l=j} (V_{j_1} \cdots V_{j_{5-l}}) (V_{k_1}^{\star} \cdots V_{k_l}^{\star}).
\]
In the above sums, we have adopted the following convention:
(1) sums are taken over all indices $j_m\in \{0,\ldots,5\}$, $k_m\in \{4,5\}$ under the constraint indicated below the sum;
(2) in the sum defining $M_{l,j}$, if $l=0$, there are no $V_{k}^\star$ terms and if $l=5$ there are no $V_{j}$ terms.
Moreover, we have $M_{l,6}=M_{l,7}=0$ for $l\geq 2$.

Differentiating the above expansion with respect to $y$, we find
\begin{align*}
\partial_y \left( (Q+V)^5-Q^5-5Q^4V \right)&= 
\sum_{j=2}^5 M_j' \theta^j + (\log\lambda) (M_5^*)' \theta^5
+\Psi_2+\Psi_3
\end{align*}
where
\begin{align*}
\Psi_2&=\sum_{j=2}^5M_j \partial_y(\theta^j) + (\log\lambda) M_5^* \partial_y( \theta^5) \\
\Psi_3&=\partial_y \left[\sum_{l=0}^5 (\log\lambda)^l \left( \sum_{j=6}^{25} M_{l,j}\theta^j\right)\right].
\end{align*}
Therefore, we obtain after rearrangement 
\begin{equation} \label{eq:W.1bis}
\begin{aligned}
\cE(W) &= - \left(\frac{\lambda_s}{\lambda}+\beta\right)\left(\Lambda Q+\Psi_{\lambda} \right) - \left( \frac{\sigma_s}{\lambda} - 1 \right) \left( Q'+\Psi_{\sigma} \right) + \Psi_W \\
& \quad +\sum_{j=2}^5 \left(F_j - (\cL A_j)'\right)\theta^j
+(\log\lambda) \left(F_5^{\star}-(\cL A_5^{\star})'\right) \theta^5
\end{aligned}
\end{equation}
where the functions $F_j$ and $F_5^\star$ are defined by
\begin{align}
F_j&=\sum_{k=1}^{j-1} c_{j-k} \Lambda_k V_k +M_j', \quad \mbox{for $j=2,3,4$,} \label{def:Fj} \\
F_5&=\sum_{k=2}^{4} c_{5-k} \Lambda_k V_k-c_1V_4^{\star}+M_5', \label{def:F5}\\
F_5^{\star}&=c_4^{\star} \Lambda_1 V_1+c_1 \Lambda_4 V_4^{\star}+(M_5^*)'\label{def:F5:star}
\end{align}
and where we have merged all the error terms in $\Psi_W = \sum_{k=1}^4 \Psi_k$, including the term
\begin{align*}
\Psi_4 & = \beta \Psi_\lambda -\sum_{j=2}^4 \left(\sum_{k=1}^{j-1} c_{j-k} \Lambda_k V_k \right)\theta^j
+\left(-\sum_{k=2}^{4} c_{5-k} \Lambda_k V_k +c_1 V_4^\star\right) \theta^5\\
&\quad - (\log\lambda)\left(c_4^{\star} \Lambda_1 V_1+c_1 \Lambda_4 V_4^{\star}\right)\theta^5.
\end{align*}

Our next objective is to cancel the second line of the identity~\eqref{eq:W.1bis} by a suitable choice of functions $A_2,\ldots,A_5,A_5^\star$ and constants $c_2,c_3,c_4^\star$.
Recall that we have fixed $c_1=-2$.
Note also that there is no error term to cancel at the order $(\log \lambda) \theta^4$, which justifies the simplified form
of the function $V_4^\star$ in~\eqref{def:V_j} (in other words, $A_4^\star=0$).

\smallskip

\noindent \textit{Construction of $A_2$.} We claim that there exists a unique even function $A_2 \in \cY$ such that $(\cL A_2)'=F_2$.
Indeed, observe from the definition of $V_1$ in~\eqref{def:V_j} that 
\begin{equation*}
 F_2=c_1 y V_1'+10 (Q^3V_1^2)'=c_1^2 \left( y A_1'+10 (Q^3A_1^2)'\right) .
\end{equation*}
Observe that $F_2$ belongs to $\cY$, is odd, and thus $y\mapsto\int_{-\infty}^{y}F_2$ is an even function that also belongs to $\cY$.
In particular $(\int_{-\infty}^{y}F_2,Q')=0$. Hence, by using (iv) of Lemma~\ref{prop:L}, there exists a unique function $A_2 \in \cY$, even, such that 
$\cL A_2=\int_{-\infty}^y F_2(y) dy$,
which implies that $(\cL A_2)'=F_2$. 

\smallskip

\noindent \textit{General computation.}
For $j \in \NN$, define
\begin{equation} \label{def:Omega}
\Omega_j= \left( \Lambda_{j-1}P+yA_1'+20(Q^3A_1P)',Q \right) .
\end{equation}
Then,
\begin{equation} \label{id:Omega}
\Omega_j= \left(\frac{5-j}2\right) m_0^2 .
\end{equation}

\noindent\emph{Proof of \eqref{id:Omega}.}
Using the definition of $\Lambda_{j-1}$ in~\eqref{def:Lambda:k},
$(P,Q)=m_0^2$ from~\eqref{P:Q}, then the cancellation $(y A_1',Q)=0$ by parity and integration by parts, we rewrite $\Omega_j$ as 
\begin{align}
\Omega_j &= \left(\frac{1-j}2P+ \Lambda P +y A_1'+20(Q^3A_1P)',Q \right) \nonumber \\
& = \frac{1-j}2 m_0^2 + \left( \Lambda P, Q \right)-20\left(Q^3A_1P,Q' \right) . \label{on:Omek}
\end{align}
Note that from \eqref{def:L} and~\eqref{def:A1}, the following identity holds
\begin{equation*}
\cL(A_1')=(\cL A_1)'+20Q^3Q'A_1=\Lambda Q+20 Q^3Q'A_1.
\end{equation*}
On the one hand, taking the scalar product of this identity with $P$, and integrating by parts, we have
\begin{equation*}
\left(\cL(A_1'),P\right)= - \left(\Lambda P ,Q \right)+20\left(Q^3A_1P,Q'\right) .
\end{equation*}
On the other hand, by integration by parts, using~\eqref{def:P} and~\eqref{def:A1},
then $(\cL P)'=\Lambda Q$ and $(A_1,\Lambda Q)=0$ by parity, we obtain
\begin{equation*}
\left(\cL(A_1'),P\right) =\left(A_1',\cL P\right) 
= -\left(\lim_{-\infty} A_1\right)\left(\lim_{-\infty} P\right) -\left(A_1,(\cL P)' \right) = -2m_0^2.
\end{equation*}
Thus,
\begin{equation} \label{id:Lambda:P:Q}
\left(\Lambda P,Q \right) - 20 \left(Q^3A_1P ,Q' \right) = 2m_0^2,
\end{equation}
which implies~\eqref{id:Omega} once inserted in~\eqref{on:Omek}.

\smallskip

\noindent \textit{Choice of $c_2$ and construction of $A_3$.} We claim that there exist $c_2 \in \RR$ and $A_3 \in \cZ_1$ such that $(\cL A_3)'=F_3$. 
Indeed, observe from~\eqref{def:V_j} and~\eqref{def:Fj} that
$F_3 =c_1c_2 F_{3,1} + c_1 F_{3,2}$
where
\begin{align*}
F_{3,1} &= \Lambda_2 P +yA_1'+20(Q^3A_1P)',\\
F_{3,2} &= \Lambda_2 A_2+20 (Q^3A_1A_2)'+10 c_1^2 (Q^2 A_1^3)'.
\end{align*}
In particular, $F_3\in \cZ_0$.
As in the construction of the profiles $P_k$ in \cite[Lemma 2.4]{CoMa1}, we look for a solution of $(\cL A_3)'=F_3$ of the form $A_3= \widetilde{A}_3-\int_{y}^{+\infty}F_3$ for some $\widetilde{A}_3 \in \cY$. Since $F_3 \in \cZ_0$, we have
$\int_{y}^{+\infty}F_3\in \cZ_1$ and this would imply that $A_3 \in \cZ_1$.
To find $\widetilde{A}_3$, by a direct computation, we observe that it should satisfy the following equation
\begin{equation} \label{on:S3}
(\cL\widetilde{A}_3)'
=F_3+\left( \cL \int_y^{\infty} F_3\right)' = S_3' ,
\end{equation}
where
$S_3 = F_3'-5Q^4\int_y^{+\infty}F_3$.
Since $S_3 \in \cY$, from (iv) of Lemma~\ref{prop:L}, the existence $\widetilde{A}_3\in \cY$
solving $\cL\widetilde{A}_3=S_3$ is guaranteed if $(S_3,Q')=0$.
A direct computation using~\eqref{on:S3} and $\cL Q'=0$ yields 
\begin{equation*}
\left(S_3,Q'\right)=-(S_3',Q)
= -\left(F_3,Q\right) - \left(\left(\cL\int_y^{+\infty}F_3\right)' ,Q \right)= -\left( F_3,Q \right) .
\end{equation*}
Moreover, recalling the definition of $\Omega_k$ in~\eqref{def:Omega}
$ ( F_3,Q ) = c_1c_2\Omega_3+c_1 (F_{3,2},Q )$.
Since $\Omega_3 \neq 0$ by~\eqref{id:Omega}, we have
\begin{equation*}
 \left( F_3,Q \right)=0 \iff c_2 = -\frac{(F_{3,2},Q)}{\Omega_3} 
\end{equation*}
which concludes the proof of the claim. 

\smallskip

\noindent \textit{Choice of $c_3$ and construction of $A_4$.}
We claim that there exist $c_3 \in \RR$ and $A_4 \in \cZ_1$ such that $(\cL A_4)'=F_4$. 
Indeed, observe from~\eqref{def:V_j} and~\eqref{def:Fj} that
$F_4 =c_1c_3 F_{4,1}+F_{4,2}$,
where 
\begin{align*} 
F_{4,1}&=\Lambda_3 P +yA_1'+20(Q^3A_1P)',\\
F_{4,2}&= c_1\Lambda_3 A_3+c_2(\Lambda_2 V_2 +20(Q^3V_1A_3)'+10(Q^3V_2^2)'+30(Q^2V_1^2V_2)'+5(QV_1^4)'.
\end{align*}
It is important to observe that $F_4 \in \cZ_0$. Indeed,
since $A_3 \in \cZ_1$, we have $\Lambda_3 A_3\in \cZ_0$ from~\eqref{VVV}.
The other terms in the definitions of $F_{4,1}$ and $F_{4,2}$ all clearly belong to $\cZ_0$.
Moreover, recalling the definition of $\Omega_k$ in~\eqref{def:Omega}, we have
$\left( F_4,Q \right) = c_1c_3\Omega_4+ \left(F_{4,2},Q \right)$.
Since $\Omega_4 \neq 0$ by~\eqref{id:Omega}, we have 
\begin{equation*}
 \left( F_4,Q \right)=0 \iff c_3 = -\frac{(F_{4,2},Q)}{c_1\Omega_4} . 
\end{equation*}
Hence, we conclude the proof of the claim arguing as for the construction of $A_3$. 

\smallskip

\noindent \textit{Construction of $A_5^{\star}$.} We claim that for any $c_4^{\star} \in \RR$,
there exists $A_5^{\star} \in \cZ_1$ such that $(\cL A_5^{\star})'=F_5^{\star}$.
Indeed, from~\eqref{def:V_j} and~\eqref{def:F5} that 
\begin{equation*}
 F_{5}^{\star}=c_1c_4^{\star}\left(\Lambda_4 P +yA_1'+20(Q^3A_1P)'\right).
\end{equation*}
In particular, using~\eqref{def:Omega} and~\eqref{id:Omega}, we have $(F_5^{\star},Q)=c_1c_4^{\star}\Omega_5=0$.
Since $F_5^{\star} \in \cZ_0$, the claim is proved as before.

\smallskip

\noindent \textit{Choice of $c_4^\star$ and construction of $A_5$.}
Since $\Omega_5=0$, the introduction of terms with $\log \lambda$ in the definition of $W$ is needed to solve the equation of $A_5$.
We claim that there exist $c_4^{\star} \in \RR$ and $A_5 \in \cZ_2$ such that $(\cL A_5)'=F_5$.
From~\eqref{def:V_j} and~\eqref{def:F5}, we compute
$F_5 = c_1c_4^{\star} P+F_{5,2}$ where 
\begin{align*}
F_{5,2}&= c_3 \Lambda_2 V_2 +c_2\Lambda_3 V_3+c_1 \Lambda_4 A_4\\ & \quad
+(20 Q^3V_1A_4 +20 Q^3V_2V_3 + 30Q^2V_1^2V_3 +30Q^2V_1V_2^2+20 QV_1^3V_2+V_1^5)'.
\end{align*}
We check that $F_5 \in \cZ_1$ from~\eqref{VVV}. Moreover, using $(P,Q)=m_0^2$, we have
\begin{equation*}
 (F_5,Q)=c_1c_4^{\star}(P,Q)+(F_{5,2},Q)=c_1c_4^{\star}m_0^2+(F_{5,2},Q) 
\end{equation*}
and so
\begin{equation*}
 \left( F_5,Q \right)=0 \iff c_4^{\star} = -\frac{(F_{5,2},Q)}{c_1m_0^2} . 
\end{equation*}

\subsection{Estimates for the components of the blow-up profile}

\begin{lemma}\label{le:2.4}
Assume~\eqref{wBT}. The following estimates hold for all $s\in \cI$ and all $y\in \RR$.
\begin{enumerate}
\item For $j\geq 1$, $q\geq 1$,
\begin{align}
&0\leq \theta^j \leq \lambda^{\frac j2}\ONE_{[-2\delta,2\delta]}(\lambda y),\label{on:theta}\\
&0\leq \lambda^{\frac j2}- \theta^j \leq \lambda^{\frac j2} \ONE_{[\delta,\infty)}(\lambda|y|), \label{def:theta3}\\
&|\partial_y^q \theta|\lesssim \delta^{-q}\lambda^{q+\frac 12} \ONE_{[\delta,2\delta]}(\lambda |y|).\label{on:dtheta}
\end{align}
\item For $j=1,\ldots,5$, $k=4,5$, $p\geq 0$, $q\geq 0$,
\begin{align} 
|(\partial_y^p V_j)\partial_y^q (\theta^j)|
&\lesssim \begin{cases} 
\lambda^{\frac j2} \omega \quad \mbox{if $q=0$ and $p>\frac{j-1}2$}\\
\delta^{-p-q} \lambda^{p+q+\frac 12} \ONE_{[-2\delta,2\delta]}(\lambda y)
\quad \mbox{otherwise};
\end{cases}\label{eq:prod}\\
|(\partial_y^p \Lambda_j V_j)\partial_y^q (\theta^j)|
&\lesssim \begin{cases} 
\lambda^{\frac j2} \omega \quad \mbox{if $q=0$ and $p>\frac{j-2}2$}\\
\delta^{-p-q} \lambda^{p+q+1} \ONE_{[-2\delta,2\delta]}(\lambda y)
\quad \mbox{otherwise};
\end{cases}\label{eq:prodL}\\
|(\partial_y^p V_k^\star)\partial_y^q (\theta^k)|+|(\partial_y^p\Lambda_k V_k^\star)\partial_y^q (\theta^k)|
&\lesssim \begin{cases} 
\lambda^{\frac k2} \omega \quad \mbox{if $q=0$ and $p>\frac{k-3}2$}\\
\delta^{-p-q}\lambda^{p+q+\frac 32} \ONE_{[-2\delta,2\delta]}(\lambda y)
\quad \mbox{otherwise.}
\end{cases}\label{eq:prod1}
\end{align}
\item For $p\geq 0$,
\begin{equation}\label{est:V0}
|\partial_y^p V(s,y)|+|\partial_y^p\Lambda V(s,y)| \lesssim \lambda^{\frac12} \omega+\delta^{-p}\lambda^{p+\frac12}\ONE_{[-2\delta,2\delta]}(\lambda y).
\end{equation}
\end{enumerate}
\end{lemma}
Note that (i) of Proposition~\ref{pr:error} is a direct consequence of \eqref{est:V0}.
We also prove (iii) of Proposition~\ref{pr:error}.
Using $W=Q+V$, we compute
$\int W^2 = \int Q^2 + 2 \int QV + \int V^2$.
From~\eqref{est:V0} and~\eqref{wBT}, we have 
$\left|\int Q V\right| \lesssim \lambda^\frac 12\lesssim \delta^2$ and
$\int V^2\lesssim \delta$, which proves~\eqref{est:asymp:mass:W}.

\begin{proof}[Proof of Lemma~\ref{le:2.4}]
Using~\eqref{def:Theta} and~$|\sigma|\leq \frac \delta4$, we have
\begin{align} 
& \mbox{if $|y|>\frac{2\delta}{\lambda}$ then $\Theta(\lambda y+\sigma) = 0$,}\label{supp:theta} \\
&\mbox{if $|y|<\frac{\delta}{\lambda}$ then $\Theta(\lambda y +\sigma)=1$ and 
$\Theta^{(p)}(\lambda y+\sigma) = 0$ for any $p\geq 1$.} \label{supp:dtheta}
\end{align}
The estimates (i) follow then from direct computations.

Let $j=1,\ldots,5$, $0\leq p\leq \frac{j-1}2$ and $q\geq 0$. From \eqref{VVV}, \eqref{on:theta}, \eqref{on:dtheta},
\begin{align*}
|(\partial_y^p V_j)\partial_y^q (\theta^j)|
&\lesssim \left( 1+|y|^{\frac{j-1}2-p} \right)\delta^{-q}\lambda^{q+\frac j2} \ONE_{[-2\delta,2\delta]}(\lambda y)\\
&\lesssim \delta^{-p-q+\frac{j-1}2}\lambda^{p+q+\frac 12} \ONE_{[-2\delta,2\delta]}(\lambda y).
\end{align*}
Let $j=1,\ldots,5$, $p>\frac{j-1}2$ and $q\geq 1$. From \eqref{VVV} and \eqref{on:dtheta}
\begin{equation*}
|(\partial_y^p V_j)\partial_y^q (\theta^j)|
\lesssim \omega(y) \delta^{-q}\lambda^{q+\frac j2} \ONE_{[\delta,2\delta]}(\lambda |y|)
\lesssim e^{-\frac{\delta}{2\lambda}}\delta^{-q}\lambda^{q+\frac j2} \ONE_{[\delta,2\delta]}(\lambda |y|).
\end{equation*}
Since $\frac{\delta}{\lambda}>1$ and $p-\frac{j-1}2>0$, we estimate $e^{-\frac{\delta}{2\lambda}}\lesssim \delta^{-p+\frac{j-1}2}\lambda^{p-\frac{j-1}2}$ and thus in this case, we also obtain
\begin{equation*}
|(\partial_y^p V_j)\partial_y^q (\theta^j)|
\lesssim \delta^{-p-q+\frac{j-1}2}\lambda^{p+q+\frac 12} \ONE_{[-2\delta,2\delta]}(\lambda y).
\end{equation*}
Let $j=1,\ldots,5$ and $p>\frac{j-1}2$. It follows from \eqref{VVV} and \eqref{on:theta} that
$|(\partial_y^p V_j) \theta^j| \lesssim \lambda^{\frac j2} \omega$.
Thus, \eqref{eq:prod} is proved. Estimates \eqref{eq:prodL} and \eqref{eq:prod1}
are proved similarly.

Last, for any $p\geq 0$, the estimate \eqref{est:V0} is a consequence of the general Leibniz rule
\[
\partial_y^p V=\sum_{j=1}^5\sum_{q=0}^p \binom{p}{q} (\partial_y^q V_j) \partial_y^{p-q} (\theta^j)
+ (\log\lambda )\sum_{k=4}^5\sum_{q=0}^p \binom{p}{q} (\partial_y^q V_k^{\star}) \partial_y^{p-q} (\theta^k)
\]
and of the previous estimates.
\end{proof}

\subsection{Estimates for the error terms}
Now, we estimate the error terms $\Psi_{\lambda}$, $\Psi_{\sigma}$ and $\Psi_W$.

\smallskip

\noindent \textit{Estimate for $\Psi_{\lambda}$.} 
Estimate~\eqref{est:Psi_lambda} for any $k\geq 0$ and for $y<0$ is a consequence of the definition of $\Psi_{\lambda}$ in~\eqref{def:Psi:lambda} and~\eqref{eq:prodL}, \eqref{eq:prod1}. For $y\geq 0$, we use the decay properties from~\eqref{VVV}.
Moreover, the same estimates show that
\[
\left| (\Psi_\lambda - (\Lambda_1 V_1) \theta, Q) \right| \lesssim \lambda.
\]
We observe that the function $\Lambda_1 V_1=c_1yA_1'$, is odd. Thus, $(\Lambda_1 V_1,Q)=0$ and it follows from \eqref{def:theta3} that 
\[ 
\left|\left( (\Lambda_1 V_1) \theta, Q \right)\right|=
\left|\left( \Lambda_1 V_1 (\theta-\lambda^{\frac12}), Q \right)\right| \lesssim \lambda^{\frac12}e^{-\frac{\delta}{2\lambda}}
\]
which yields \eqref{est:Psi_lambda:Q}.

\smallskip

\noindent \textit{Estimate for $\Psi_{\sigma}$.} It follows from the definition of $\Psi_{\sigma}$ in~\eqref{def:Psi:sigma} and~\eqref{VVV}, \eqref{eq:prod}, \eqref{eq:prod1} that~\eqref{est:Psi_sigma} holds for any $k\geq 0$.

\smallskip

\noindent \textit{Estimate for $\Psi_{W}$.} First, we claim that, for any $k \ge 0$,
\begin{equation} \label{eq:psitheta}
\left| \partial_y^k\Psi_1 \right|+\left| \partial_y^k\Psi_2 \right|
\lesssimD \lambda^{\frac72+k} \ONE_{[-2\delta,2\delta]}(\lambda y).
\end{equation}
We decompose 
$\Psi_1=\Psi_{1,1}+\Psi_{1,2}+\Psi_{1,3}+\Psi_{1,4}$ according to the four terms in~\eqref{def:Psi:theta1}.
Using the expression of $Q$, \eqref{VVV} and \eqref{on:dtheta}, we have for any $k\geq 0$,
\[
|\partial_y^k\Psi_{1,1}|+|\partial_y^k\Psi_{1,2}|
\lesssimD \omega(y)\lambda^{\frac 32} \ONE_{[\delta,2\delta]}(\lambda |y|)
\lesssimD e^{-\frac\delta{2\lambda}} \lambda^{\frac 32} \ONE_{[-2\delta,2\delta]}(\lambda y).
\]
For all the terms in~$\partial_y^k\Psi_{1,3}$, we use~\eqref{eq:prod} with $p+q=3+k$ and $q \ge 1$, which proves that 
$|\partial_y^k\Psi_{1,3}|\lesssimD \lambda^{\frac72+k} \ONE_{[-2\delta,2\delta]}(\lambda y).$
Using~\eqref{eq:prod1}, we have $|\partial_y^k\Psi_{1,4}|\lesssimD\left|\log \lambda\right|
\lambda^{\frac92+k} \ONE_{[-2\delta,2\delta]}(\lambda y).$
These estimates justify~\eqref{eq:psitheta} for the term $\Psi_1$.

We decompose $\Psi_2 = \Psi_{2,1}+\Psi_{2,2}$, where 
\[
\Psi_{2,1} = \sum_{j=2}^4 M_j \partial_y(\theta^j) + \widetilde M_5 \partial_y(\theta^5)+ (\log \lambda) M_5^\star \partial_y(\theta^5),
\quad \Psi_{2,2} = V_1^5 \partial_y(\theta^5)
\]
and $\widetilde M_5 = M_5 - V_1^5$.
We observe that for $j=2,\ldots,4$, $k\geq 0$,
$|\partial_y^kM_j|+|\partial_y^k\widetilde M_5|+|\partial_y^kM_5^\star|\lesssim \omega$.
In particular, using \eqref{on:dtheta}, we have
$|\partial_x^k\Psi_{2,1}|\lesssimD e^{-\frac\delta{2\lambda}} \lambda^{2} \ONE_{[-2\delta,2\delta]}(\lambda y)$.
Moreover, using~\eqref{on:dtheta}
and $|\partial_y^kV_1|\lesssim 1$, we have $|\partial_y^k\Psi_{2,2}|\lesssimD \lambda^{\frac72+k} \ONE_{[-2\delta,2\delta]}(\lambda y)$.
This proves~\eqref{eq:psitheta} for $\Psi_2$.

Second, we prove that, for all $k \ge 0$,
\begin{equation} \label{eq:psiV}
\left| \partial_y^k\Psi_3 \right| \lesssimD 
\left|\log\lambda\right|\lambda^{3}\omega + \lambda^{\frac72+k} \ONE_{[-2\delta,2\delta]}(\lambda y).
\end{equation}
We claim that, for any $p\geq 0$,
\[
|\partial_y^p M_{l,j}| \lesssim
\begin{cases} \omega & \mbox{if $p + l > \frac{j}2 - \frac 52 $,}\\
 1 + |y|^{\frac{j}2 - \frac 52 - p - l} & \mbox{otherwise.}
\end{cases}
\]
Indeed, after applying the Leibniz rule, any term of the sum expressing $\partial_y^p M_{l,j}$ is the product of five functions
$\partial_y^{p_m} V_{j_m}$ or $\partial_y^{p_m}V_{k_m}^\star$; moreover
if $p + l > \frac{j}2 - \frac 52 $ then there exists at least one $j_m\in \{1,\ldots,5-\ell\}$ such that
$p_m> \frac{j_m}2-\frac 12$ and so $\partial_y^{p_m} V_{j_m}\in \cY$ by \eqref{VVVbis}, or 
 one $k_m\in \{1,\ldots,\ell\}$ such that
$p_m+1> \frac{j_m}2-\frac 12$ and so $\partial_y^{p_m} V_{k_m}^\star\in \cY$ by \eqref{VVVbis}.

Therefore, recalling that $M_{l,6}=M_{l,7}=0$ for $l\geq 2$ in the sum defining $\Psi_3$,
the estimate~\eqref{eq:psiV} is proved as~(ii)-(iii) of Lemma~\ref{le:2.4}.

Last, we claim, for all $k \ge 0$,
\begin{equation} \label{eq:psibeta}
\left|\partial_y^k \Psi_4 \right| \lesssimD \lambda^{3+k} \left|\log \lambda\right| |y| \ONE_{[-2\delta,0]}(\lambda y) + \lambda^{3} \left|\log \lambda\right|\omega
\end{equation}
We compute from~\eqref{def:beta}, \eqref{def:Psi:lambda},
\begin{align*}
\Psi_4&= \sum_{j=6}^{8}\left( \sum_{k=j-3}^{5} c_{j-k}\Lambda_k V_k \right)\theta^j
-\sum_{j=6}^{8} \left(\sum_{k=\max(4,j-3)}^5 c_{j-k} V_k^{\star} \right) \theta^j \\
& \quad +(\log\lambda) \sum_{j=6}^{8}\left(\sum_{k=\max(4,j-3)}^5c_{j-k}\Lambda_k V_k^{\star} \right)\theta^j 
-(\log\lambda)c_4^{\star}\sum_{j=8}^9 V_{j-4}^{\star} \theta^j \\
& \quad 
+(\log\lambda)c_4^{\star}\sum_{j=6}^9 (\Lambda_{j-4} V_{j-4}) \theta^j 
+(\log\lambda)^2c_4^{\star}\sum_{j=8}^9 (\Lambda_{j-4}V_{j-4}^{\star}) \theta^j.
\end{align*}
By \eqref{VVV}, all the functions involved in the above expression of $\Psi_4$ belong to $\cZ_1$.
Moreover, the minimal power of $\theta$ is $\theta^6$ with $\log \lambda$, and $\theta^8$ with higher powers of $\log\lambda$.
Thus, estimate~\eqref{eq:psibeta} follows from \eqref{on:theta} and \eqref{on:dtheta}.

\subsection{Variation of the energy}\label{s.variation}
By direct computation
\begin{equation*}
\lambda^2 \frac{d}{ds}\left[ \frac {E(W)}{\lambda^2}\right]
=-\int \partial_sW \left( \partial_y^2W+W^5 \right) -2\frac{\lambda_s}{\lambda} E(W)
\end{equation*}
and thus, using~\eqref{eq:W.1}-\eqref{eq:W.2} and integrating by parts,
\begin{align*}
\lambda^2 \frac{d}{ds}\left[ \frac {E(W)}{\lambda^2}\right]
&=-\frac{\lambda_s}{\lambda} \int \Lambda W \left( \partial_y^2W+W^5 \right)-2\frac{\lambda_s}{\lambda} E(W) \\ & \quad +\left( \frac{\sigma_s}{\lambda}-1\right)\int \left( Q'+\Psi_{\sigma} \right) \left( \partial_y^2W+W^5\right) \\ & \quad +\int \left( \frac{\lambda_s}{\lambda}+\beta\right) \left( \Lambda Q+\Psi_{\lambda} \right) \left( \partial_y^2W+W^5\right) \\ & \quad 
-\int \Psi_W \left( \partial_y^2 W+W^5 \right) \\ & 
:= e_1+e_2+e_3+e_4.
\end{align*}
First, integrating by parts,
$-\int \Lambda W \left( \partial_y^2W+W^5 \right) = 2E(W)$ and thus $e_1=0$.

\smallskip 
\noindent \textit{Estimate for $e_2$}. It follows from~\eqref{eq:Q}, \eqref{est:V0}, \eqref{est:Psi_sigma} 
and $\int QQ'=0$ that
\begin{equation*}
\left| e_2\right| \le \left| \frac{\sigma_s}{\lambda}-1\right|
\left(\int \left( |Q'|+|\Psi_{\sigma}| \right) \left( |\partial_y^2V|+|(Q+V)^5-Q^5|\right)+\int |\Psi_{\sigma}|Q \right)
\lesssimD \lambda^{\frac12}\left| \frac{\sigma_s}{\lambda}-1\right| . 
\end{equation*}

\smallskip 
\noindent \textit{Estimate for $e_3$}. The function $\beta$ depending on $y$, we rewrite $e_3$ 
using $\widetilde \beta$ (see~\eqref{def:beta_tilde}) as
\begin{align*}
e_3
&=\left( \frac{\lambda_s}{\lambda}+\widetilde{\beta}\right)\int \left( \Lambda Q+\Psi_{\lambda} \right) \left( \partial_y^2W+W^5\right) 
+\int \left( \beta-\widetilde{\beta} \right)\left( \Lambda Q+\Psi_{\lambda} \right) \left( \partial_y^2W+W^5\right) \\ 
&:=e_{3,1}+e_{3,2} .
\end{align*}
On the one hand, we deduce from~\eqref{eq:Q}, \eqref{est:V0}, \eqref{est:Psi_lambda} and $\int (\Lambda Q)Q =0$ that 
\begin{equation*}
\left| e_{3,1} \right| \lesssim \left| \frac{\lambda_s}{\lambda}+\widetilde{\beta} \right|\left( \int \left( |\Lambda Q|+|\Psi_{\lambda}| \right) \left( |\partial_y^2V|+|(Q+V)^5-Q^5|\right)+\int |\Psi_{\lambda}|Q \right)
\lesssimD \lambda^{\frac12}\left| \frac{\lambda_s}{\lambda}+\widetilde{\beta} \right| .
\end{equation*}
On the other hand, by using \eqref{def:beta}-\eqref{def:beta_tilde} 
\begin{equation*}
 e_{3,2}
 = \int \left( \sum_{j=1}^4 c_j\left(\theta^j-\lambda^{\frac{j}2}\right)+c_4^{\star} (\log\lambda) \left(\theta^4-\lambda^2 \right) \right)\left( \Lambda Q+\Psi_{\lambda} \right) \left( \partial_y^2W+W^5\right),
\end{equation*}
and so it follows from \eqref{def:theta3}, \eqref{est:W0} and \eqref{est:Psi_lambda} that
\begin{equation*} 
\left| e_{3,2} \right|
\lesssimD\lambda^{\frac12} \int \left(\omega+\lambda\ONE_{[-2\delta,2\delta]}(\lambda y)\right)
\left(\omega+\lambda^{\frac 52}\ONE_{[-2\delta,2\delta]}(\lambda y)\right)
\ONE_{[\delta,\infty]}(\lambda|y|) \lesssimD \lambda^3 .
\end{equation*}
Hence
\begin{equation*}
\left| e_{3} \right| \lesssimD \lambda^{\frac12}\left| \frac{\lambda_s}{\lambda}+\widetilde{\beta} \right|+ \lambda^3 .
\end{equation*}

\smallskip

\noindent \textit{Estimate for $e_4$}. It follows from~\eqref{est:W0} and~\eqref{est:Psi_W} that 
\begin{align*} 
\left| e_4 \right| &\lesssimD \int \left(\lambda^{3}\left|\log\lambda\right|\left( |y| \ONE_{[-2\delta,0]}(\lambda y)+\omega\right)
+\lambda^{\frac72} \ONE_{[0,2\delta]}(\lambda y) \right) \left( \omega+\lambda^{\frac52}\ONE_{[-2\delta,2\delta]}(\lambda y) \right)\\
&\lesssimD \lambda^3 \left| \log\lambda \right| .
\end{align*}
We complete the proof of~\eqref{est:deriv:energy:W} gathering these estimates.

\smallskip

The proof of Proposition~\ref{pr:error} is complete.
From now on, when we consider a blow-up profile $W$ such as in \eqref{def:V}-\eqref{def:W},
we take $c_1=-2$, the constants $c_2$, $c_3$, $c_4^{\star} \in \RR$ and the functions 
$A_2,A_3,A_4,A_5,A_5^{\star}$ defined in Proposition~\ref{pr:error}.

\subsection{Formal blow-up law}\label{S:2.3}

Proposition~\ref{pr:error} (ii) suggests to look for a blow-up profile
$W$ with parameters $(\lambda,\sigma)$ satisfying formally
\begin{equation*}
\frac{\lambda_s}{\lambda} + \widetilde\beta =0,\quad \frac{\sigma_s}{\lambda} - 1 =0.
\end{equation*}
The replacement of $\beta(s,y)$ by $\widetilde\beta(s)$ in the equation of $\lambda$
is justified by~\eqref{def:theta3}.
In view of the expression of $\widetilde\beta$ in \eqref{def:beta_tilde} (recall $c_1=-2$), we introduce
\begin{equation} \label{def:G}
G(\lambda) = \int^{\lambda_0}_\lambda \frac{d\eta}{\eta^{\frac 32} \big(2 - c_2 \eta^\frac 12 - c_3 \eta - c_4^{\star} \eta^\frac 32 \log\eta\big)}.
\end{equation}
This function is well-defined for any $0<\lambda\leq \lambda_0$, where $\lambda_0>0$ is small enough.
Rewriting the equation for $\lambda$ as $\lambda_s G'(\lambda)=-1$,
yields by integration
$G(\lambda(s))= G(\lambda(0))-s = -s$.
We have $G(\lambda) = \lambda^{-\frac 12} - \frac {c_2}4 \log \lambda + \cO(1) $ as $ \lambda \downarrow 0 $.
Moreover, the function $G$ is decreasing and one-to-one from $(0,\lambda_0]$ to $[0,\infty)$.
Thus, $\lambda(s) = G^{-1}(-s) \sim s^{-2}$.
Next, we integrate $\sigma_s = \lambda$ choosing $\sigma(0)=\int_{-\infty}^0 \lambda$.
Thus, for any $s<-1$, we obtain the estimates
\begin{equation*}
\left| \lambda(s) - \frac 1{s^2}\right|\lesssim \frac {\log|s|} {|s|^3},\quad
\left|\sigma(s)- \frac 1s\right| \lesssim \frac {\log|s|}{s^2}.
\end{equation*}

\section{Modulation close to the blow-up profile}\label{S:3}
\subsection{Refined blow-up profile} 
In~\cite{MaMeRa1}, a refined approximate blow-up profile around $Q$, involving a small parameter $b$ and
denoted by $Q_b$,
allows to obtain an additional decisive orthogonality condition on the residual term.
We proceed similarly for the blow-up profile~$W$.
For a $\cC^1$ real-valued function $b:\cI\to \RR$ such that
\begin{equation}\label{wBTb}
|b| < \delta^2,
\end{equation}
we define the function $W_b(s,y)=W_b(y;\lambda(s),\sigma(s),b(s))$ by
\begin{equation} \label{def:W:b}
W_b(s,y) = W(s,y) + b(s) P_b(s,y)
\end{equation}
where
\begin{equation} \label{def:P:b}
P_b(s,y) = \chi_b(y) P(y),\quad \chi_b(s,y) = \chi(|b(s)|^\gamma y),\quad \gamma = \frac 34, 
\end{equation}
and the function $\chi$ is defined in \eqref{def:chi}.
The next lemma provides estimates for the refined profile $W_b$.
Let
\begin{equation} \label{def:Psi:b}
\Psi_b= \partial_y \left( b\partial_y^2 P_b - b P_b + W_b^5-W^5 \right) + b \Lambda Q + b \beta \Lambda P_b.
\end{equation}

\begin{lemma} Assume \eqref{wBT} and \eqref{wBTb}. For all~$s\in\cI$.
\begin{enumerate}
\item \emph{Pointwise estimates for $W_b$.} For any $p\geq 1$, for any $y \in \RR$,
\begin{align} 
 |W_b|+|\Lambda W_b| &\lesssim \omega+\lambda^{\frac12}\ONE_{[-2\delta,2\delta]}(\lambda y)+|b| \ONE_{[-2,0]}(|b|^{\gamma}y), \label{est:W0:b} \\ 
 |\partial_y^pW_b|+|\partial_y^p\Lambda W_b| &\lesssim \omega+\delta^{-p}\lambda^{p+\frac12} \ONE_{[-2\delta,2\delta]}(\lambda y)+|b|^{1+\gamma p}\ONE_{[-2,-1]}(|b|^{\gamma}y).\label{est:W12:b}
\end{align}
\item \emph{Error term in the equation of $W_b$.}
For $p=0,1$, for all $y \in \RR$,
\begin{equation} \label{est:Psi0:b}
|\partial_y^p \Psi_b| \lesssimD (b^2+|b|\lambda^{\frac12})\omega 
+|b|\lambda^{\frac12} \ONE_{[-2,0]}(|b|^{\gamma}y)+ |b|^{1+\gamma(1+p)}\ONE_{[-2,-1]}(|b|^{\gamma}y).
\end{equation}
Moreover, 
\begin{equation} \label{est:Psib:Q}
 \left|(\Psi_b,Q)-2c_1m_0^2b\lambda^{\frac12} \right| \lesssimD b^2+|b| \lambda .
\end{equation}
\item \emph{Mass and Energy of $W_b$.}
\begin{align}
 \left| \int W_b^2-\int W^2 \right| &\lesssimD |b|+|b|^{1-\gamma} \lambda^{\frac12}, \label{mass:W:b}\\ 
 \left| {E(W_b)}- {E(W)}\right| &\lesssimD |b|+|b|^{1-\gamma}\lambda^{\frac52} . \label{energy:W:b}
\end{align}
\end{enumerate}
\end{lemma}

\begin{proof} (i) The estimates \eqref{est:W0:b}-\eqref{est:W12:b} follow directly from the estimates \eqref{est:W0} for $W$ and the definitions of $\chi$, $W_b$ and $P_b$ in \eqref{def:chi}, \eqref{def:W:b} and \eqref{def:P:b}. 

\smallskip

(ii) Expanding $W_b=W+bP_b$ and using $(\cL P)' = \Lambda Q$ in the definition of $\Psi_b$, we find 
\begin{equation}\label{eq:Psib}
\begin{aligned}
\Psi_b & = b (1-\chi_b) \Lambda Q 
+ b \beta \Lambda P_b\\
&\quad + 3 b (\partial_y \chi_b) P'' + 3 b (\partial_y^2 \chi_b) P'
+b (\partial_y^3 \chi_b) P - b (\partial_y \chi_b) P + 5 b(\partial_y \chi_b) Q^4 P
\\&\quad + b\partial_y \left( 5 (W^4-Q^4) P_b + 10 b W^3 P_b^2 +10 b^2 W^2 P_b^3 + 5 b^3 W P_b^4 + b^4 P_b^5 \right).
\end{aligned}
\end{equation}
We estimate separately each term on the right-hand side of \eqref{eq:Psib}. First, from the properties of $\chi_b$ and $\Lambda Q \in \cY$, for $p=0,1$,
\begin{equation*} 
\left|b\partial_y^p \left((1-\chi_b) \Lambda Q\right) \right| \lesssim |b|e^{-\frac 34|y|} \ONE_{(-\infty,-1]}(|b|^{\gamma}y) \lesssim |b|e^{-\frac 14 |b|^{-\gamma}} \omega .
\end{equation*}
Next, we compute 
\begin{equation} \label{Lambda:P:b}
 \Lambda P_b = (\Lambda P)\chi_b+y(\partial_y \chi_b) P,
\end{equation}
and we estimate
\begin{equation} \label{Lambda:P:b.1}
\left|\Lambda P_b \right| \lesssim \omega+\ONE_{[-2,0]}(|b|^{\gamma}y) \quad \text{and} \quad \left|\partial_y\Lambda P_b \right| \lesssim \omega+|b|^{\gamma}\ONE_{[-2,-1]}(|b|^{\gamma}y) .
\end{equation}
Thus, it follows from \eqref{def:beta} and \eqref{on:theta} that
\begin{align*}
 \left|b \beta \Lambda P_b \right|
 &\lesssim |b| \lambda^{\frac12}\omega +|b| \lambda^{\frac12} \ONE_{[-2,0]}(|b|^{\gamma}y), \\
 \left|b \beta \partial_y(\Lambda P_b) \right|
 &\lesssim |b| \lambda^{\frac12}\omega+|b|^{1+\gamma} \lambda^{\frac12} \ONE_{[-2,-1]}(|b|^{\gamma}y).
 \end{align*}
Next, we deduce from \eqref{def:P:b} and $P \in \cZ_0$ that, for $p=0,1$, 
\begin{align*}
&\left| b \partial_y^p\left((\partial_y \chi_b) P''\right)\right| +\left| b \partial_y^p\left((\partial_y^2 \chi_b) P'\right) \right|
+\left| b\partial_y^p\left((\partial_y \chi_b) Q^4 P\right) \right|\lesssim e^{-\frac 14|b|^{-\gamma}} \omega, \\
&\left|b \partial_y^p\left((\partial_y^3 \chi_b) P\right) \right|+\left|b \partial_y^p\left((\partial_y \chi_b) P\right) \right|
\lesssim |b|^{1+(1+p)\gamma}\ONE_{[-2,-1]}(|b|^{\gamma}y) .
\end{align*}
Recall from~\eqref{def:W},
$W^4-Q^4=4Q^3V+6Q^2V^2+4QV^3+V^4$.
By~\eqref{est:V0}, for $p=0,1,2$,
\begin{align*}
& |\partial_y^p (Q^3 V)|+ |\partial_y^p (Q^2 V^2)| + |\partial_y^p (Q V^3)|\lesssim \lambda^\frac12 \omega,\\
& |\partial_y^p (V^4)|\lesssim \lambda^2 \omega + \delta^{-p} \lambda^{p+2}.
\end{align*}
Thus, for $p=0,1$,
\begin{equation*}
\big|b \partial_y^{1+p} \big((W^4-Q^4)P_b\big)\big|
\lesssimD |b|\lambda^{\frac12} \omega +|b|\lambda^{3+p} \ONE_{[-2,0]}(|b|^{\gamma}y)\\
+ |b|^{1+\gamma}\lambda^2 (\lambda^p+|b|^{p\gamma})\ONE_{[-2,-1]}(|b|^{\gamma}y).
\end{equation*}
Last, by using \eqref{est:W0} and \eqref{def:P:b}, we have, for $p=0,1$,
\begin{align*} 
\left| b^2\partial_y^{1+p}\left(W^3 P_b^2\right) \right|
&\lesssimD |b|^2\omega+|b|^2\lambda^{\frac52+p} \ONE_{[-2,0]}(|b|^{\gamma}y)
+|b|^{2+\gamma}\lambda^{\frac 32}(\lambda^p+|b|^{p\gamma}) \ONE_{[-2,-1]}(|b|^{\gamma}y), \\
\left| b^3\partial_y^{1+p}\left(W^2 P_b^3\right) \right|
&\lesssimD |b|^3\omega+|b|^3\lambda^{2+p} \ONE_{[-2,0]}(|b|^{\gamma}y)
+|b|^{3+\gamma}\lambda(\lambda^p+|b|^{p\gamma}) \ONE_{[-2,-1]}(|b|^{\gamma}y), \\
\left| b^4\partial_y^{1+p}\left(W P_b^4\right) \right|
&\lesssimD |b|^4\omega+|b|^4\lambda^{\frac32+p} \ONE_{[-2,0]}(|b|^{\gamma}y)
+|b|^{4+\gamma}\lambda^{\frac12}(\lambda^p+|b|^{p\gamma})\ONE_{[-2,-1]}(|b|^{\gamma}y),\\
\left| b^5 \partial_y^{1+p}(P_b^5) \right|
&\lesssimD |b|^5\omega+|b|^{5+(1+p)\gamma}\ONE_{[-2,-1]}(|b|^{\gamma}y).
\end{align*}
We deduce~\eqref{est:Psi0:b} by gathering these estimates. 

Proof of~\eqref{est:Psib:Q}. 
Taking the scalar product of the expression of $\Psi_b$ in \eqref{eq:Psib} with $Q$
and using the same estimates as before,
we observe that the two terms $b \beta \Lambda P_b$ and $20b\partial_y (Q^3V_1\theta P)$ give contributions of size $b\lambda^{\frac 12}$
while all the other terms give contributions of size at most $b^2+|b|\lambda$.
First, we observe
\begin{align*}
\left|\left(\beta \Lambda P_b-c_1\lambda^{\frac12}\Lambda P, Q\right) \right|
&\lesssim \left|\left(\beta (P_b- P),\Lambda Q\right) \right|
+|(\beta - \widetilde \beta) P, \Lambda Q)|
+|\widetilde \beta - c_1\lambda^{\frac12}| |( P, \Lambda Q)|\\
&\lesssim \lambda^{\frac12} e^{-\frac 14|b|^{-\gamma}} + \lambda^{\frac12}e^{-\frac{\delta}{4\lambda}}+ \lambda 
\lesssimD \lambda^{\frac 12} b^2 + \lambda.
\end{align*}
Second, we compute from \eqref{def:V_j} and \eqref{def:theta3} that 
\begin{equation*}
 \left|\left(\partial_y(Q^3V_1\theta P),Q\right)
+c_1\lambda^{\frac12}\big(Q^3A_1P,Q'\big)\right| \lesssim \lambda^{\frac 12} e^{-\frac\delta{4\lambda}}\lesssimD \lambda.
\end{equation*}
Thus, \eqref{est:Psib:Q} follows from the identity \eqref{id:Lambda:P:Q}.

(iii) Expanding $W_b=W+bP_b$, we obtain
\begin{equation*}
 \int W_b^2=\int W^2+2b \int WP_b+b^2\int P_b^2 .
\end{equation*}
Observe from \eqref{est:W0} and \eqref{def:P:b} that 
\begin{equation*}
\left| b \int WP_b \right| \lesssim |b|+|b|^{1-\gamma}\lambda^{\frac12} ,\quad
b^2 \int P_b^2 \lesssim |b|^{2-\gamma}.
\end{equation*}
Hence, \eqref{mass:W:b} is proved.
Expanding $W_b=W+bP_b$ in the definition of the energy, we obtain 
\begin{equation*} 
E(W_b)=E(W)+b\int (\partial_yW) (\partial_yP_b)+\frac12b^2\int (\partial_yP_b)^2 -\frac16 \int \left( (W+bP_b)^6-W^6 \right).
\end{equation*}
Observe from \eqref{est:W0} and \eqref{def:P:b} that 
\begin{gather*}
\left| b\int (\partial_yW) (\partial_yP_b) \right| \lesssim |b| ,\quad
b^2 \int (\partial_yP_b)^2 \lesssim b^2, \\
\left|\int \left( (W+bP_b)^6-W^6 \right) \right| \lesssim 
|b|+|b|^{1-\gamma}\lambda^{\frac52} .
\end{gather*}
Hence, \eqref{energy:W:b} is proved. 
\end{proof}

\subsection{Modulation close to the blow-up profile}\label{S:3.2}
Let 
\[
\mbox{$\cI=[S,s_0]$ where $S<s_0<0$ and $|s_0| \gg 1$.}
\]
We look for $(w(s,y),\lambda(s),\sigma(s))$ solution of
the rescaled equation~\eqref{rescaled} on $\cI$ 
such that $w$ has the form
\[
w(s,y) = W_{b}(s,y) + \varepsilon(s,y), 
\]
where the parameters $\lambda$, $\sigma$, $b$ satisfy \eqref{wBT}, \eqref{wBTb}
and the function $\varepsilon$ satisfies
\begin{equation}\label{eq:H1}
\|\varepsilon(s)\|_{H^1} \leq \delta^\frac 14
\end{equation}
and the orthogonality relations
\begin{equation}\label{ortho}
(\varepsilon(s),\Lambda Q)=(\varepsilon(s),y\Lambda Q)=(\varepsilon(s),Q)=0.
\end{equation}
Moreover, we assume that 
\begin{equation} \label{w:0}
w(S) \equiv W(S) \iff b(S)=0, \, \varepsilon(S)\equiv 0.
\end{equation}
In this context, the existence and uniqueness of $\cC^1$ functions $\lambda$, $\sigma$, $b$ ensuring the orthogonality relations~\eqref{ortho} follow from standard arguments based on the implicit function theorem.
In the case where the blow-up profile is $Q$, this is justified in~\cite[Lemma~2.5]{MaMeRa1} and the references therein.
The same proof applies to the refined ansatz $W$.
The orthogonality relations \eqref{ortho} are taken from~\cite{MaMejmpa} (as in the series of papers~\cite{CoMa1,CoMa2,MaMeNaRa,MaMeRa1,MaMeRa2,MaMeRa3,MaPi,Mjams}) in order to ensure the positivity of a quadratic form
related to a virial argument (see Lemma~\ref{le:virial}).

We derive the equation of $\varepsilon$. First, we check from \eqref{rescaled} that
\begin{equation}\label{eq:eps} 
\partial_s \varepsilon + \partial_y \left[\partial_y^2\varepsilon-\varepsilon+\left((W_b+\varepsilon)^5-W_b^5\right)\right]
-\frac{\lambda_s}{\lambda} \Lambda\varepsilon - \left( \frac{\sigma_s}{\lambda}-1\right) \partial_y \varepsilon
+ \cE_b(W) = 0
\end{equation}
where (see \eqref{eq:W.1} and \eqref{def:Psi:b} for the definitions of $\cE(W)$ and $\Psi_b$)
\begin{equation*}
\cE_b(W) = \cE(W) - b \Lambda Q 
-\left(\frac{\lambda_s}{\lambda}+\beta\right) \Lambda (bP_b)
- \left( \frac{\sigma_s}{\lambda} - 1 \right) (b\partial_y P_b)
+ \partial_s (bP_b)+ \Psi_b.
\end{equation*}
Next, using \eqref{eq:W.2}, we rewrite
\begin{equation}\label{def:EbW}
\cE_b(W) = - \vec{m} \cdot \vec{M}Q+ \Psi_M + \Psi_{W} + \Psi_b, 
\end{equation}
where
\begin{equation} \label{def:mM}
 \vec m =\begin{pmatrix} \frac{\lambda_s}{\lambda}+\widetilde\beta+b \\ \frac{\sigma_s}{\lambda}-1 \end{pmatrix},
 \quad \vec M =\begin{pmatrix} \Lambda \\ \partial_y \end{pmatrix},
\end{equation}
and 
\begin{equation}\label{def:PsiM}
\begin{split}
\Psi_M &= -\left(\frac{\lambda_s}{\lambda}+\widetilde\beta+b\right) \left( \Psi_{\lambda}+b\Lambda P_b\right)
- \left( \frac{\sigma_s}{\lambda} - 1 \right) (\Psi_{\sigma}+b \partial_y P_b) \\ & \quad+(\widetilde\beta-\beta)\left(\Lambda Q+\Psi_{\lambda}+b\Lambda P_b\right)+b\left(\Psi_{\lambda}+b\Lambda P_b\right)+ b_s \left(\chi_b + \gamma y \partial_y \chi_b \right) P.
\end{split}
\end{equation}

\begin{lemma}[Modulation estimates] 
For any $s \in \cI$,
\begin{align}
|\vec m|:=\left| \frac{\lambda_s}{\lambda} +\widetilde\beta +b\right|+
\left| \frac{\sigma_s}{\lambda} - 1 \right| \lesssimD \|\varepsilon\|_{L^2_\loc}+\lambda^3\left|\log \lambda\right|+|b|\lambda^{\frac12}+b^2,
 \label{est:lambda:sigma}\\
\left| b_s+2c_1b\lambda^{\frac12}\right| \lesssimD \|\varepsilon\|_{L^2_\loc}^2+\|\varepsilon\|_{L^2_\loc} (\lambda^{\frac12}+|b|)+\lambda^3\left|\log \lambda\right|+|b|\lambda+b^2 .\label{est:b_s}
\end{align}
\end{lemma}
\begin{proof}
By using the definition of $\cL$ in \eqref{def:L}, we rewrite the equation of $\varepsilon$ as
\begin{equation}\label{eq:eps.2} 
\partial_s \varepsilon + \partial_y\cL \varepsilon
-\vec{m} \cdot \vec{M}(Q+\varepsilon)+(\widetilde\beta+b)\Lambda \varepsilon+\partial_y\cR_1+\partial_y\cR_2+\Psi_M+\Psi_W+\Psi_b=0
\end{equation}
where 
\begin{equation*}
\cR_1=(W_b+\varepsilon)^5-W_b^5-5W_b^4\varepsilon ,\quad 
\cR_2=5(W_b^4-Q^4)\varepsilon .
\end{equation*}
We start by general estimates on the error terms in \eqref{eq:eps.2}.
We claim that, for any $f \in \cY$, 
\begin{align} 
\left|\left(\partial_y \cR_1,f\right) \right|
&\lesssim \|\varepsilon\|_{L^2_\loc}^2 , \label{est:mod.1}\\
\left|\left(\partial_y \cR_2,f\right) \right|
&\lesssim \left(\lambda^{\frac12}+|b|\right)\|\varepsilon\|_{L^2_\loc}, \label{est:mod.2}\\
\left|\left(\Psi_{M},f\right) \right|
& \lesssimD |b_s|+(|\vec{m}|+b) (\lambda^{\frac12}+|b|)+e^{-\frac{\delta}{4\lambda}}, \label{est:mod.3}\\
\left|\left(\Psi_{W},f\right) \right| & \lesssimD \lambda^3 \left|\log \lambda\right|, \label{est:mod.4}\\ 
\left|\left(\Psi_{b},f\right) \right| &\lesssimD |b|(\lambda^{\frac12}+|b|). \label{est:mod.5}
\end{align}
where the implicit constants depend on $f$.
Indeed, the estimate~\eqref{est:mod.1} is a consequence of 
$|\cR_1|\lesssim |W_b|^3 \varepsilon^2 + |\varepsilon|^5\lesssim |\varepsilon|^2$ by~\eqref{est:W0:b} and
$\|\varepsilon\|_{L^\infty}^2\lesssim \|\partial_y \varepsilon\|_{L^2}\|\varepsilon\|_{L^2}\lesssim 1$.
The estimate~\eqref{est:mod.2} follows from integration by parts $\left(\partial_y \cR_2,f\right)= -\left((W_b^4-Q^4)\varepsilon, f'\right)$
and $|W_b^4-Q^4| \lesssim \lambda^{\frac 12}+|b|$
using \eqref{est:V0} and \eqref{def:W:b}.
The estimates~\eqref{est:mod.3}-\eqref{est:mod.5} follow from \eqref{def:theta3}, \eqref{est:Psi_lambda}, \eqref{est:Psi_sigma}, \eqref{est:Psi_W}, \eqref{est:Psi0:b}, \eqref{Lambda:P:b} and \eqref{def:PsiM}. 

Now, using the first orthogonality relation in \eqref{ortho}, \eqref{eq:eps.2} and $(Q',\Lambda Q)=0$, we obtain 
\begin{align*}
\left|\left( \frac{\lambda_s}{\lambda}+\widetilde\beta+b\right)+
\frac{(\varepsilon,\cL(\Lambda Q)')}{\|\Lambda Q\|_{L^2}^2} \right|
&\lesssimD (|\vec{m}|+ \|\varepsilon\|_{L^2_\loc}) \left(\lambda^{\frac12}+|b|+ \|\varepsilon\|_{L^2_\loc}\right)+\left|\left(\partial_y \cR_1,\Lambda Q\right) \right|\\ 
&\quad +\left|\left(\partial_y \cR_2,\Lambda Q\right) \right|
 +\left|\left(\Psi_M,\Lambda Q\right) \right|
+\left|\left(\Psi_{W},\Lambda Q\right) \right|+\left|\left(\Psi_b,\Lambda Q\right) \right| .
\end{align*}
Hence, using \eqref{est:mod.1}-\eqref{est:mod.5} with $f=\Lambda Q$,
\begin{equation} \label{est:lambda.2}
\left| \frac{\lambda_s}{\lambda}+\widetilde\beta+b \right| \lesssimD
\|\varepsilon\|_{L^2_\loc}+ |\vec{m}| \left(\lambda^{\frac12}+|b|+ \|\varepsilon\|_{L^2_\loc}\right)+|b_s| +\lambda^3\left|\log \lambda\right|+|b|\lambda^{\frac12}+b^2 .
\end{equation}
Similarly, by the second orthogonality relation in \eqref{ortho} and $(Q',y\Lambda Q)=\|\Lambda Q\|_{L^2}^2$, we obtain
\begin{equation} \label{est:sigma.2}
\left| \frac{\sigma_s}{\lambda}-1\right|
\lesssimD
\|\varepsilon\|_{L^2_\loc}+ |\vec{m}| \left(\lambda^{\frac12}+|b|+ \|\varepsilon\|_{L^2_\loc}\right)+|b_s| +\lambda^3\left|\log \lambda\right|+|b|\lambda^{\frac12}+b^2 .
\end{equation}

Last, using the third orthogonality relation in \eqref{ortho} and \eqref{eq:eps.2}, we have
\begin{align*}
0=(\partial_s\varepsilon,Q )&=\left( \partial_y\cL\varepsilon,Q\right)-\left(\vec{m} \cdot\vec M (Q+\varepsilon) , Q \right)+(\widetilde \beta+b) (\Lambda \varepsilon,Q)\\ 
&\quad + \left(\partial_y \cR_1, Q\right) +\left(\partial_y \cR_2, Q\right)
+(\Psi_M,Q)+(\Psi_W,Q)+(\Psi_b,Q) .
\end{align*}
We observe the special cancellations
$\left(\partial_y\cL\varepsilon,Q\right)=-\left( \varepsilon,\cL Q'\right)=0$ and $(\Lambda \varepsilon,Q)=-(\varepsilon,\Lambda Q)=0$.
Moreover, it follows from the identities $(\Lambda Q, Q)=(Q',Q)=0$ that 
\begin{equation*} 
\left| \left(\vec{m} \cdot\vec M (Q+\varepsilon) , Q \right) \right| \lesssimD |\vec{m}| \|\varepsilon\|_{L^2_\loc}.
\end{equation*}
Recalling $m_0^2=(P,Q)>0$ (see \eqref{P:Q}) and arguing as for the proof of \eqref{est:mod.3} but using \eqref{est:Psi_lambda:Q} instead of \eqref{est:Psi_lambda}, we obtain
\begin{equation*} 
\left|(\Psi_M,Q)-m_0^2 b_s \right| \lesssimD |\vec{m}|(\lambda^{\frac12}+|b|)+|b|\lambda+b^2+e^{-\frac{\delta}{\lambda}}. 
\end{equation*}
Then, we deduce combining these estimates with \eqref{est:Psib:Q} and \eqref{est:mod.1}-\eqref{est:mod.4} with $f=Q$ that 
\begin{equation} \label{est:b_s.2}
\left| b_s+2c_1b\lambda^{\frac12}\right| 
\lesssimD (|\vec{m}|+\|\varepsilon\|_{L^2_\loc}) \left(\lambda^{\frac12}+|b|+ \|\varepsilon\|_{L^2_\loc}\right)
+\lambda^3\left|\log \lambda\right|+|b|\lambda+b^2 .
\end{equation}

Therefore, we conclude the proof of \eqref{est:lambda:sigma} combining \eqref{est:lambda.2}, \eqref{est:sigma.2} and \eqref{est:b_s.2},
and taking~$\delta$ small enough. The proof of \eqref{est:b_s} then follows from \eqref{est:lambda:sigma} and \eqref{est:b_s.2}. 
\end{proof}

\section{Bootstrap setting} \label{bootstrap:sec}

\subsection{Definition of a sequence of solutions}
The time variables $t$ and $s$ will be approximately related by 
$5t=(-s)^{-5}$ (see \eqref{time})
which motivates the definitions, for $n\geq 1$ large,
\[
T_n = \frac 15 \frac 1{n^5},\quad
S_n = - n.
\]
We denote by $W_{b,n}(t,y)$ the function $W_b(y;\lambda_n(t),\sigma_n(t),b_n(t))$ defined in Section~\ref{S:3}
for $\lambda_n(t)$, $\sigma_n(t)$ and~$b_n(t)$ to be chosen.
We define the solution $u_n$ of \eqref{gkdv} with initial data at $T_n$
\[
u_n(T_n) = \frac 1{\lambda_n^\frac 12(T_n)} W_{b,n}\left(T_n, \frac{x-\sigma_n(T_n)}{\lambda_n(T_n)}\right)
\]
with the choices
\begin{equation} \label{at_Tn}
\lambda_n(T_n)=G^{-1}(n) ,\quad
\sigma_n(T_n) = \frac1{n},\quad
b_n(T_n)= 0. 
\end{equation}
Recall from Section~\ref{S:2.3} the definition of $G$ and $|\lambda_n(T_n) - n^{-2}| \lesssim n^{-3}(\log n)$.

We consider the solution $u_n$ for times $t\geq T_n$. As long as it exists and remains close to $Q$ up to rescaling and translation,
we can decompose it as in Section \ref{S:3.2}
\[
u_n(t,x) = \frac 1{\lambda_n^{\frac 12}(t)} \left( W_{b,n}\left(t, \frac{x-\sigma_n(t)}{\lambda_n(t)}\right)
+ \varepsilon_n\left(t,\frac{x-\sigma_n(t)}{\lambda_n(t)}\right)\right),
\]
where $\varepsilon_n$ satisfies \eqref{ortho}.
At $t=T_n$, this decomposition satisfies \eqref{at_Tn} and $\varepsilon_n(T_n)\equiv 0$.

We define the rescaled time variable $s$
\begin{equation}\label{defts}
s=s(t) = S_n + \int_{T_n}^t \frac{dt'}{\lambda_n^3(t')}.
\end{equation}
From now on, any time-dependent function will be seen either as a function of $t$ or as a function of $s$.

\subsection{Notation for the energy-virial functional}

We follow the notation and presentation of \cite[Section 3.4]{MaMeRa1} and \cite[Section 3.1]{CoMa2}.
Let $\psi,\varphi\in\cC^\infty$ be non decreasing functions such that
\begin{equation*}
\psi(y) = \left\{ \begin{aligned} e^y \quad & \mbox{for } y < -1, \\ 1 \quad & \mbox{for } y > -\frac 12, \end{aligned} \right.
\qquad \mbox{and} \qquad
\varphi(y) = \left\{ \begin{aligned} e^y \quad & \mbox{for } y < -1, \\ y + 1 \quad & \mbox{for } -\frac 12 < y < \frac 12, \\
2 - e^{-y}\quad & \mbox{for } y > 1. \end{aligned} \right.
\end{equation*}
We note that such functions satisfy $\frac 12 e^y \leq \psi(y) \leq 3e^y$
and $\frac 13 e^y \leq \varphi(y) \leq 3e^y$ for $y<0$,
and $\frac 12 \varphi(y) \leq \psi(y) \leq 3\varphi(y)$ for all $y\in\RR$.
Moreover, we may choose the function $\varphi$ such that $\frac 14 \leq \varphi'(y) \leq 1$ for $y\in[-1,1]$
and so $\frac 13 e^{-|y|} \leq \varphi'(y) \leq 3e^{-|y|} $ for all $y\in\RR$.

For $B > 100$ large to be chosen later we define
\[
\psi_B(y) = \psi\left( \frac yB \right) \quad\mbox{and}\quad
\varphi_B(y) = \varphi\left( \frac yB\right).
\]
From the properties of $\psi$ and $\varphi$, we have for all $y<0$,
\begin{equation} \label{cut:left}
\frac 12 e^{\frac yB} \leq \psi_B(y) \leq 3 e^{\frac yB}, \quad
\frac 13 e^{\frac yB} \leq \varphi_B(y) \leq 3 e^{\frac yB},
\end{equation}
and for all $y\in\RR$,
\begin{equation} \label{cut:onR}
\left\{
\begin{gathered}
\frac 1{3} e^{-\frac{|y|}{B}} \leq B \varphi_B'(y) \leq 3e^{-\frac {|y|}B} \leq 9\varphi_B(y), \\
\frac 12 \varphi_B(y) \leq \psi_B(y) \leq 3\varphi_B(y), \\[2mm]
\psi_B'(y) + B^2|\psi_B'''(y)| + B^2|\varphi_B'''(y)| \lesssim \varphi_B'(y).
\end{gathered}
\right.
\end{equation}
We set
\[
\cN_B(\varepsilon_n) = \left(\int \left( (\partial_y \varepsilon_n)^2 + \varepsilon_n^2 \right) \varphi_B \right)^{\frac 12}
\]
and we observe 
\begin{equation} \label{eL2B}
\|\varepsilon_n\|_{L^2_\loc} \lesssim \left( B \int \varepsilon_n^2\varphi_B' \right)^{\frac 12} \lesssim \cN_B(\varepsilon_n).
\end{equation}

\subsection{Bootstrap estimates}
Let $C_1^{\star}>1$ and $C_2^{\star}>1$ be two large constants to be fixed later
(possibly depending on $\delta$).
We introduce the bootstrap estimates
\begin{equation} \label{BS}
\begin{aligned}
\left| \lambda_n(s) - \frac 1{s^2}\right|&\leq C_1^{\star} |s|^{-3} \log|s|,\\
\left|\sigma_n(s)-\frac 1{|s|}\right|&\leq C_1^{\star} |s|^{-2} \log|s|, \\
|b_n(s)|&\leq C_1^{\star} |s|^{-5} \log |s|,
\end{aligned}
\end{equation}
and 
\begin{equation}\label{BS2}
\cN_B(\varepsilon_n(s)) +
\left(\int_{S_n}^s \left(\frac\tau s\right)^4 \|\varepsilon_n(\tau)\|_{L^2_\loc}^2\, d\tau \right)^{\frac 12} 
\leq C_2^{\star} |s|^{-5} \log |s|.
\end{equation}
For $s_0<-1$, $|s_0|$ large enough to be chosen later, we define $S_n^\star \in (S_n,s_0]$ by 
\begin{equation*}
 S_n^{\star}=\sup \left\{\mbox{$S_n<\tilde{s}<s_0$ : \eqref{eq:H1}, \eqref{BS} and~\eqref{BS2} are satisfied for all $s\in [S_n,\tilde{s}]$}\right\}.
\end{equation*}
\begin{proposition} \label{prop:bootstrap}
There exist $C_1^{\star}>1$, $C_2^{\star}>1$, $B>1$ and $s_0<-1$, independent of $n$, such that for $n$ large enough, $S_n^{\star}=s_0$. 
\end{proposition}

Proposition~\ref{prop:bootstrap} is proved in the rest of Section~\ref{bootstrap:sec} and in Section~\ref{Sec:5}.
We work on the time interval $\cI=[S_n,S_n^{\star}]$ where the bootstrap estimates~\eqref{eq:H1}, \eqref{BS} and~\eqref{BS2} hold.
In every step of the proof, the value of $|s_0|$ will be taken large enough, depending on the constants $\delta$, $C_1^\star$, $C_2^\star$
and~$B$, but independent of $n$.
For simplicity of notation, we drop the $n$ in index of the functions $(W_{b,n},\varepsilon_n,\lambda_n,\sigma_n,b_n)$.

\subsection{Closing the parameter estimates}
The next lemma follows from inserting \eqref{BS}-\eqref{BS2} into the estimates~\eqref{est:lambda:sigma} and \eqref{est:b_s},
using also \eqref{eL2B}.
\begin{lemma}
For all $s\in\cI$,
\begin{align}
 |\vec m |= \left| \frac{\lambda_s}{\lambda} +\widetilde \beta + b \right|+\left| \frac{\sigma_s}{\lambda} - 1 \right|
 &\lesssimD \|\varepsilon\|_{L^2_\loc} + C_1^{\star} |s|^{-6}\log |s| , \label{cons:BS.1}\\
 \left| b_s + 4 s^{-1} b \right| & \lesssimD |s|^{-1}\|\varepsilon\|_{L^2_\loc} + |s|^{-6}\log |s|. \label{cons:BS.2} 
\end{align}
\end{lemma}

We strictly improve the bootstrap estimates~\eqref{BS} on the parameters. 
\begin{lemma} \label{lemma:param:BS}
There exists $C_1^{\star}>1$, independent of $C_2^{\star}$
such that for all $s \in \cI$,
\begin{equation} \label{BS:improved}
\begin{aligned}
\left| \lambda(s) - \frac 1{s^2}\right|&\leq \frac{C_1^{\star}}2 |s|^{-3} \log|s|, \\ 
\left|\sigma(s)-\frac 1{|s|}\right|&\leq \frac{C_1^{\star}}2 |s|^{-2} \log|s|,\\ 
|b(s)|&\leq \frac{C_1^{\star}}2 |s|^{-5} \log |s|.
\end{aligned}
\end{equation}
\end{lemma} 

\begin{proof}
First, from~\eqref{cons:BS.1} and~\eqref{eL2B}-\eqref{BS}-\eqref{BS2}, we have
$\left| \lambda_s G'(\lambda) + 1 \right|\lesssim |s|^{-2}$,
where the function $G$ is defined in \eqref{def:G}.
Integrating on $[S_n,s]$ for any $s\in \cI$ and using \eqref{at_Tn}, we obtain
$|G(\lambda(s)) - |s|| \lesssim |s|^{-1}$.
Since $G(\lambda) = \lambda^{-\frac 12} - \frac {c_2}4 \log \lambda + \cO(1) $ as $ \lambda \downarrow 0 $, we obtain
\begin{equation}\label{numero}
\lambda(s) = |s|^{-2} + c_2 |s|^{-3} \log |s| + \cO(|s|^{-3}),
\end{equation}
which implies the estimate for $\lambda$ in \eqref{BS:improved} by choosing $C_1^{\star} \ge 4c_2$ and $|s_0|$ large enough.

Second, from \eqref{cons:BS.1}, \eqref{BS2} and \eqref{numero}, we have
$| \sigma_s - |s|^{-2} | \lesssim |s|^{-3} \log |s|$.
By integration, using $\sigma(S_n) = |S_n|^{-1}$,
\[
| \sigma(s) - |s|^{-1} | \leq c |s|^{-2} \log |s|
\]
for some $c\geq 0$.
Choosing $C_1^{\star} \ge 4c$, this implies the estimate for $\sigma$ in \eqref{BS:improved}. 

Last, from \eqref{cons:BS.2} and \eqref{BS}-\eqref{BS2},
$\ | (s^4 b)_s \ | \lesssimD |s|^{3} \|\varepsilon\|_{L^2_\loc} + |s|^{-2}\log |s|$,
and using $b(S_n)=0$, 
\begin{equation*}
|s|^4 |b(s)| \lesssimD \int_{S_n}^s |\tau|^3 \|\varepsilon(\tau)\|_{L^2_\loc} d\tau+ |s|^{-1}\log |s|. 
\end{equation*}
Using \eqref{eL2B} and then \eqref{BS2}, we observe that
\[ 
|\tau|^3\|\varepsilon(\tau)\|_{L^2_\loc}
\lesssimD |\tau|^3[\cN_B(\varepsilon(\tau))]^{\frac23}\|\varepsilon(\tau)\|_{L^2_\loc}^{\frac13}
\lesssimD (C_2^{\star})^{\frac23}|\tau|^{-1}(\log|\tau|)^{\frac23}\cdot |\tau|^{\frac23}\|\varepsilon(\tau)\|_{L^2_\loc}^{\frac13}.
\]
By the H\"older inequality and then again \eqref{BS2}, it follows that 
\begin{equation*}
\begin{aligned}
\int_{S_n}^s |\tau|^3 \|\varepsilon(\tau)\|_{L^2_\loc} d\tau 
&\lesssimD
(C_2^{\star})^{\frac23} \left(\int_{S_n}^s |\tau|^{-\frac65}(\log|\tau|)^{\frac45} d\tau \right)^{\frac56}\left(\int_{S_n}^s |\tau|^4 \|\varepsilon(\tau)\|_{L^2_\loc}^2 d\tau \right)^{\frac16} \\ 
& \lesssimD C_2^{\star}\left(|s|^{-\frac16}(\log|s|)^{\frac23}\right)\left(|s|^{-1}(\log|s|)^{\frac13} \right)
= C_2^{\star}|s|^{-1-\frac16}\log(|s|) .
\end{aligned}
\end{equation*}
Thus, $|b(s)| \leq c_\delta |s|^{-5}\log |s|$
which yields the estimate for $b$ in \eqref{BS:improved} by choosing $C_1^{\star} \ge 4c_\delta$.
\end{proof}
The constant $C_1^\star$ in \eqref{BS} is now fixed according to Lemma~\ref{lemma:param:BS} and it is not tracked anymore.

\subsection{Control of global norms from the conservation of mass and energy}

\begin{lemma}\label{le:mass-ener}
For any $s \in \cI$,
\begin{equation}\label{GN:BS}
\|\varepsilon(s)\|_{L^2} \lesssim \delta^{\frac 12} ,\quad
\|\partial_y \varepsilon(s)\|_{L^2}\lesssim \delta^{-\frac 12}s^{-2}, \quad
\|\varepsilon(s)\|_{L^{\infty}}\lesssim |s|^{-1} .
\end{equation}
\end{lemma}
For $|s_0|$ large enough and $\delta$ small, this improves strictly the bootstrap~\eqref{eq:H1}.
\begin{proof} 
From~\eqref{w:0} and $\int w^2(S)=\int w^2(s)$, it follows that
\begin{equation*}
 \int W^2(S)= \int w^2(s)=\int (W_b+\varepsilon)^2(s)=\int W_b^2(s)+2\int (W+bP_b)(s) \varepsilon(s)+\int \varepsilon^2(s) .
\end{equation*}
Using \eqref{est:asymp:mass:W}, we have $\left| \int W^2(s)-\int W^2(S) \right| \lesssim \delta$.
Moreover, we deduce from the orthogonality relation $\int \varepsilon Q =0$,
 the Cauchy-Schwarz inequality and the estimates \eqref{est:V0} and \eqref{def:P:b} that
\begin{equation*}
 \left| \int (W+bP_b) \varepsilon \right| \le \left( \|V\|_{L^2}+|b|\|P_b\|_{L^2}\right) \|\varepsilon \|_{L^2} \le \frac12 \int \varepsilon^2+c(\delta+|\sigma|+\lambda^{\frac12}+|b|^{2-\gamma}) .
\end{equation*}
Combining these estimates with \eqref{mass:W:b} comparing $W$ and $W_b$, yields 
the control of the $L^2$ norm of $\varepsilon$ (recall that~\eqref{wBT} and \eqref{wBTb}
control the size of the parameters in terms of $\delta$).

\smallskip

The conservation of energy implies that 
$\lambda^{-2}(S)E(w(S))=\lambda^{-2}(s)E(w(s)$. Thus, by~\eqref{w:0}
\begin{equation*}
 \frac{E(W(S))}{\lambda^2(S)} = \frac{E(w(s))}{\lambda^2(s)}=
 \frac{E(W_b(s)+\varepsilon(s))}{\lambda^2(s)} .
\end{equation*}
Using the identity $\int(Q''+Q^5)\varepsilon=\int Q\varepsilon=0$,
\begin{equation*}
\begin{split}
E(W_b+\varepsilon)&=E(W_b)+\int (\partial_yV+b \partial_y P_b)(\partial_y\varepsilon) +\frac12 \int (\partial_y\varepsilon)^2 \\ & \quad -\int (W_b^5-Q^5)\varepsilon -\frac16 \int \left((W_b+\varepsilon)^6-W_b^6-6W_b^5\varepsilon \right) .
\end{split}
\end{equation*}
Using the Cauchy-Schwarz inequality and then \eqref{est:V0},
\begin{equation*}
 \left|\int (\partial_yV)(\partial_y\varepsilon) \right|
 \leq \|\partial_y V\|_{L^2}\|\partial_y\varepsilon\|_{L^2}
 \leq \frac18 \|\partial_y\varepsilon\|_{L^2}^2 + c \delta^{-1}\lambda^2.
\end{equation*}
Next, the Cauchy-Schwarz inequality and \eqref{def:P:b} yield 
\begin{equation*}
 \left|\int b(\partial_y P_b)(\partial_y\varepsilon) \right| \lesssim |b|\|\partial_y\varepsilon\|_{L^2}
 \le \frac18 \|\partial_y\varepsilon\|_{L^2}^2+c|b|^2 .
\end{equation*}
We observe from \eqref{est:V0}, \eqref{def:P:b},
\begin{align*}
\left| \int (W_b^5-Q^5) \varepsilon \right|
&\lesssim \int Q^4 (|V|+|b||P_b|)|\varepsilon|+\int (|V|^5+|b|^5|P_b|^5)|\varepsilon|\\
& \lesssim (\lambda^{\frac12}+|b|)\|\varepsilon\|_{L^2_\loc}+(\lambda^2\sqrt{\delta}+|b|^{5-\frac{\gamma}2})\|\varepsilon\|_{L^2}
\lesssim (\lambda^{\frac12}+|b|)\|\varepsilon\|_{L^2_\loc} + \delta\lambda^2 + |b|^2
\end{align*}
and from \eqref{est:W0:b} and the Gagliardo-Nirenberg inequality \eqref{sharpGN},
\begin{align*}
\left| \int \left((W_b+\varepsilon)^6-W_b^6-6W_b^5\varepsilon \right)\right|
&\lesssim \int \left( W_b^4\varepsilon^2+\varepsilon^6 \right)\\
&\lesssim \|\varepsilon\|_{L^2_\loc}^2+(\lambda^2+|b|^4)\|\varepsilon\|_{L^2}^2+\|\varepsilon\|_{L^2}^4\|\partial_y \varepsilon\|_{L^2}^2\\
&\lesssim \|\varepsilon\|_{L^2_\loc}^2+ \delta \lambda^2 + |b|^4 + \delta^2 \|\partial_y \varepsilon\|_{L^2}^2.
\end{align*}
Gathering these estimates, for $\delta$ small enough, we deduce
\begin{align*}
 \|\partial_y \varepsilon\|_{L^2}^2 &\lesssim 
 \|\varepsilon\|_{L^2_\loc}^2+\lambda^{\frac12}\|\varepsilon\|_{L^2_\loc}+\delta^{-1}\lambda^2+b^2 +\lambda^2 \left| \frac{E(W_b)}{\lambda^2}-\frac{E(W(S))}{\lambda^2(S)}\right| .
\end{align*}
By \eqref{energy:W:b} and then \eqref{wBT} we have (recall $\gamma = 3/4$)
\[
|E(W_b)-E(W)|\lesssim |b| + |b|^{\frac 14}\lambda^{\frac 52}\lesssim |b| + \lambda^{\frac {10}3}
\lesssim |b| +\delta \lambda^2.
\]
From \eqref{est:deriv:energy:W}, \eqref{cons:BS.1}, \eqref{eL2B}, \eqref{BS} and \eqref{BS2}, we have
\[
\left| \frac{d}{dt} \left[ \frac{E(W)}{\lambda^2}\right]\right| \lesssim C_2^\star |s|^{-2}\log|s|.
\]
Thus, by integration,
\begin{equation}\label{est:asymp:energy:W}
\left| \frac{E(W(s))}{\lambda^2(s)}-\frac{E(W(S))}{\lambda^2(S)}\right| \lesssim  C_2^\star |s|^{-1}\log|s| \lesssim 1.
\end{equation}
This yields, using also \eqref{eL2B}, \eqref{BS} and \eqref{BS2},
\[
\|\partial_y\varepsilon\|_{L^2}^2 
\lesssim \delta^{-1}\lambda^2 + \|\varepsilon\|_{L^2_\loc}^2+\lambda^{\frac12}\|\varepsilon\|_{L^2_\loc}+|b|
\lesssim \delta^{-1} s^{-4}.
\]
Last, the estimate $\|\varepsilon\|_{L^{\infty}}^2\lesssim \|\varepsilon\|_{L^2}\|\partial_y\varepsilon\|_{L^2}$
completes the proof of~\eqref{GN:BS}.
\end{proof}

\section{Energy estimates}\label{Sec:5}

To complete the proof of Proposition \ref{prop:bootstrap},
we improve the bootstrap estimate~\eqref{BS2} using the following variant
of the virial-energy functional introduced in~\cite{MaMeRa1}
\[
\cF = \frac{1}{\lambda^2}
\left[ \int (\partial_y\varepsilon)^2 \psi_B + \int \varepsilon^2 \varphi_B
- \frac13 \int \left( (W_b+\varepsilon)^6 - W_b^6 - 6 W_b ^5 \varepsilon\right) \psi_B\right].
\]

\subsection{Virial-energy estimates}

\begin{proposition}\label{pr:local-energy}
There exist $\mu_1,\mu_2>0$ and $B>100$ such that the following hold on $\cI$.
\begin{enumerate}
\item \emph{Time variation of $\cF$.} 
\begin{equation}\label{time:local}
\frac{d\cF}{ds} + \frac{\mu_1} {\lambda^2} \int \left( (\partial_y\varepsilon)^2 + \varepsilon^2\right)\varphi_B'
\lesssimD C_2^{\star}|s|^{-7}(\log|s|)^2 .
\end{equation}
\item \emph{Coercivity of $\cF$.} 
\begin{equation}\label{coer:local}
\cF \geq \frac{\mu_2}{\lambda^2} \left[{\cN}_B(\varepsilon)\right]^2 .
\end{equation}
\end{enumerate}
\end{proposition}

\begin{proof}[Proof of Proposition~\ref{pr:local-energy}]
Proof of \eqref{time:local}. By one integration by parts and then \eqref{eq:eps}
\begin{align*}
\lambda^2 \frac{d\cF}{ds}& = 
2 \int \partial_s\varepsilon \left(-\psi_B' \partial_y\varepsilon - \psi_B \partial_y^2\varepsilon + \varepsilon \varphi_B
- \psi_B \left[ (W_b+\varepsilon)^5 - W_b^5\right] \right) \\
&\quad - 2 \frac{\lambda_s}{\lambda}\lambda^2 \cF
- 2 \int (\partial_s W_b) \left[ (W_b+\varepsilon)^5 - W_b^5 - 5 W_b^4 \varepsilon \right] \psi_B\\
& = f_1+f_2+f_3+f_4+f_5
\end{align*}
where
\begin{align*}
f_1 & = 2 \int \partial_y\left(-\partial_y^2\varepsilon + \varepsilon - \left[ (W_b+\varepsilon)^5 - W_b^5\right]\right) G_B(\varepsilon)\\
f_2 & = - 2 \int \cE_b(W) G_B(\varepsilon) \\
f_3 & = 2 \left(\frac{\sigma_s}{\lambda} - 1\right) \int (\partial_y\varepsilon) G_B(\varepsilon) \\
f_4 & = 2 \frac{\lambda_s}{\lambda} \int \left((\Lambda \varepsilon) G_B(\varepsilon)
-(\partial_y\varepsilon)^2 \psi_B - \varepsilon^2 \varphi_B
+ \frac 13 \left[ (W_b+\varepsilon)^6 - W_b^6 - 6 W_b^5 \varepsilon\right]\psi_B \right)\\
f_5 & = - 2 \int (\partial_s W_b) \left[ (W_b+\varepsilon)^5 - W_b^5 - 5 W_b^4 \varepsilon \right] \psi_B
\end{align*}
and
\[
G_B(\varepsilon) = - \partial_y(\psi_B \partial_y\varepsilon) + \varepsilon \varphi_B
- \psi_B \left[ (W_b+\varepsilon)^5 - W_b^5\right].
\]

\smallskip

\noindent \emph{Estimate for $f_1$.} We claim that there exist $\mu_1>0$, $B_0 \ge 100$, such that for all $B \ge B_0$,
\begin{equation} \label{est:f1}
 f_{1} +2\mu_1 \int \left(\varepsilon^2+(\partial_y\varepsilon)^2\right)
 \varphi_B' \lesssimD |s|^{-13} .
\end{equation}

Following the computations page 89 in \cite{MaMeRa1}
\begin{align*}
f_1 & = - \int \left[ 3 (\partial_y^2 \varepsilon)^2 \psi_B'
+(3 \varphi_B'+\psi_B'-\psi_B''') (\partial_y \varepsilon)^2
+ (\varphi_B'-\varphi_B''') \varepsilon^2 \right]\\
&\quad - \frac 13 \int \left[ (W_b+\varepsilon)^6 - W_b^6 - 6 (W_b+\varepsilon)^5 \varepsilon \right] (\varphi_B'-\psi_B')\\
&\quad + 2 \int\left[ (W_b+\varepsilon)^5 - W_b^5 - 5 W_b^4 \varepsilon\right] (\partial_yW_b) (\psi_B-\varphi_B)\\
&\quad + 10 \int \psi_B' (\partial_y \varepsilon) \left[ (\partial_y W_b)\left( (W_b+\varepsilon)^4 - W_b^4\right)
+ (W_b+\varepsilon)^4 (\partial_y \varepsilon)\right]\\
&\quad - \int \psi_B' \left[\left(-\partial_y^2\varepsilon+ \varepsilon- \left( (W_b+\varepsilon)^5 -W_b^5\right)\right)^2
- (-\partial_y^2 \varepsilon + \varepsilon)^2 \right].
\end{align*}
We decompose $f_1  =:f_1^{<}+f_1^{\sim}+f_1^{>}$
where $f_1^{<}$, $f_1^{\sim}$ and $f_1^{>}$ correspond to the integration over the regions $y<-\frac{B}2$, $|y|\le \frac{B}2$ and $y>\frac{B}2$ respectively. 

We first estimate $f_1^{<}$. By using the properties of $\varphi_B$ and $\psi_B$ in \eqref{cut:left}-\eqref{cut:onR} and choosing $B \ge 100$ large enough, we have 
\begin{align*}
 f_1^{<}+3\int_{y<-\frac{B}2} (\partial_y^2\varepsilon)^2& \psi_B'+\frac12 \int_{y<-\frac{B}2} \left( \varepsilon^2+(\partial_y\varepsilon)^2 \right) \varphi_B'
 \\ & \lesssim \int_{y<-\frac{B}2} \left(W_b^4\varepsilon^2+\varepsilon^6\right) \varphi_B'
 +B\int |\partial_yW_b| \left( |W_b|^3\varepsilon^2+|\varepsilon|^5\right) \varphi_B'\\ & 
 \quad +\int_{y<-\frac{B}2} |\partial_y\varepsilon| \left(|\partial_yW_b| \left(|W_b|^3|\varepsilon|+|\varepsilon|^4\right)+|\partial_y\varepsilon|\left( |W_b|^4+\varepsilon^4\right)\right)\varphi_B' \\ 
 &\quad +\int_{y<-\frac{B}2} \left| -2\left(-\partial_y^2\varepsilon+\varepsilon \right)+ \left( (W_b+\varepsilon)^5-W_b^5\right) \right| \left|(W_b+\varepsilon)^5-W_b^5 \right|\psi_B' \\ 
 &=:f_{1,1}^{<}+f_{1,2}^{<}+f_{1,3}^{<}+f_{1,4}^{<} .
\end{align*}
It follows from \eqref{est:W0:b}, then \eqref{cut:left}-\eqref{cut:onR} and \eqref{GN:BS}
\begin{equation*}
f_{1,1}^{<} \lesssim \int_{y<-\frac{B}2}\varepsilon^2 \omega \varphi_B'+\left(|s|^{-4}+\|\varepsilon\|_{L^{\infty}}^4\right)\int_{y<-\frac{B}2} \varepsilon^2 \varphi_B' \lesssim \left(e^{-\frac{B}4}+|s|^{-4}\right)\int_{y<-\frac{B}2} \varepsilon^2 \varphi_B' .
\end{equation*}
Similarly, \eqref{est:W0:b}, \eqref{est:W12:b}, then \eqref{cut:left}-\eqref{cut:onR} and \eqref{GN:BS} yield 
\begin{equation*}
f_{1,2}^{<} \lesssim B \int_{y<-\frac{B}2}\varepsilon^2 \omega \varphi_B'
+B\left(|s|^{-6}+|s|^{-3}\|\varepsilon\|_{L^{\infty}}^3\right)\int_{y<-\frac{B}2} \varepsilon^2 \varphi_B'
\lesssim \left(e^{-\frac{B}8}+|s|^{-6}\right)\int_{y<-\frac{B}2} \varepsilon^2 \varphi_B' ,
\end{equation*}
\begin{align*}
 f_{1,3}^{<} &\lesssim \int_{y<-\frac{B}2} |\partial_yW_b| \left(|W_b|^3+|\varepsilon|^3\right) \left(\varepsilon^2+(\partial_y\varepsilon)^2\right)\varphi_B'+\int_{y<-\frac{B}2} (\partial_y\varepsilon)^2\left(|W_b|^4+\varepsilon^4 \right)\varphi_B' \\ 
 & \lesssim \left(e^{-\frac{B}{4}}+|s|^{-4}\right) \int_{y<-\frac{B}2}\left(\varepsilon^2+(\partial_y\varepsilon)^2\right)\varphi_B'
\end{align*}
and 
\begin{align*}
 f_{1,4}^{<} 
 & \lesssim \int_{y<-\frac{B}2} \left(|W_b|^4+\varepsilon^4\right) \left(\varepsilon^2+(\partial_y^2\varepsilon)^2\right) \psi_B'+
 \int_{y<-\frac{B}2}\left(|W_b|^8+\varepsilon^8\right)\varepsilon^2\varphi_B'
 \\
 & \lesssim \left(e^{-\frac{B}{4}}+|s|^{-4}\right) \int_{y<-\frac{B}2} \left((\partial_y^2\varepsilon)^2\psi_B'+
\varepsilon^2\varphi_B' \right).
\end{align*}
Gathering these estimates and taking $B>100$ and $|s|$ large enough, we deduce that 
\begin{equation} \label{est:f1<}
f_1^{<}+\int_{y<-\frac{B}2} (\partial_y^2\varepsilon)^2 \psi_B'+\frac14 \int_{y<-\frac{B}2} \left( \varepsilon^2+(\partial_y\varepsilon)^2 \right) \varphi_B' \le 0 .\end{equation}

Since $\psi_B' \equiv 0$ on $y>\frac{B}2$, we observe using also \eqref{cut:onR} that, for $B>100$ large enough, 
\begin{align*}
f_1^> +\frac12 \int_{y>\frac{B}2} \left( \varepsilon^2+(\partial_y\varepsilon)^2\right) \varphi_B' 
\lesssim \int_{y>\frac{B}2}\left(|W_b|^4+\varepsilon^4 \right) \varepsilon^2 \varphi_B' + \int_{y>\frac{B}2}|\partial_yW_b|\left( |W_b|^3+|\varepsilon|^3 \right)\varepsilon^2 \varphi_B.
\end{align*}
Using \eqref{est:W0:b}, \eqref{est:W12:b} and \eqref{GN:BS}, it follows that
\begin{equation*}
 f_{1}^{>} \lesssim \int_{y>\frac{B}2} \varepsilon^2 \omega 
 +|s|^{-4}\int_{y>\frac{B}2} \varepsilon^2 \varphi_B
 \lesssim e^{-\frac{B}8} \int_{y>\frac{B}2} \varepsilon^2 \varphi_B'+|s|^{-4}[\cN_B(\varepsilon)]^2.
\end{equation*}
Thus, using \eqref{BS2}, for $B$ and $|s|$ large enough,
we obtain
\begin{equation} \label{est:f1>}
f_1^{>}+\frac14 \int_{y>\frac{B}2} \left( \varepsilon^2+(\partial_y\varepsilon)^2\right) \varphi_B' \lesssim |s|^{-13} .
\end{equation}
 
In the region $|y|<\frac{B}2$, using $\varphi_B(y)=1+\frac{y}B$ and $\psi_B(y)=1$,
\begin{align*} 
f_1^{\sim}&=-\frac1B \int_{|y|<\frac{B}2} \left(3(\partial_y\varepsilon)^2+\varepsilon^2+\frac13\left[ (W_b+\varepsilon)^6 - W_b^6 - 6 (W_b+\varepsilon)^5 \varepsilon \right]\right) \\
&\quad -\frac2B \int_{|y|<\frac{B}2}\left[ (W_b+\varepsilon)^5 - W_b^5 - 5 W_b^4 \varepsilon\right] y(\partial_yW_b).
\end{align*}
We decompose $f_1^{\sim}$ as
\begin{align*} 
f_1^{\sim}
&=-\frac1B \int_{|y|<\frac{B}2}\left(3(\partial_y\varepsilon)^2+\varepsilon^2-5Q^4\varepsilon^2+20yQ^3Q'\varepsilon^2\right) \\ 
&\quad -\frac1{3B}\int_{|y|<\frac{B}2}\left[ (W_b+\varepsilon)^6 - W_b^6 - 6 (W_b+\varepsilon)^5 \varepsilon + 15Q^4\varepsilon^2\right] \\ & 
\quad -\frac2B\int_{|y|<\frac{B}2}\left[ (W_b+\varepsilon)^5 - W_b^5 - 5 W_b^4 \varepsilon-10W_b^3\varepsilon^2\right] y(\partial_yW_b) \\ 
& \quad -\frac{20}B \int_{|y|<\frac{B}2} y\left( W_b^3\partial_yW_b-Q^3Q'\right) \varepsilon^2 \\ 
&=:f_{1,1}^{\sim}+f_{1,2}^{\sim}+f_{1,3}^{\sim}+f_{1,4}^{\sim}.
\end{align*}

To control the main quadratic term $f_{1,1}^{\sim}$, we rely on the virial-type estimate proved in \cite{MaMeRa1} (see also \cite[Proposition 4]{MaMejmpa}).
\begin{lemma}[{\cite[Lemma 3.4]{MaMeRa1}}]\label{le:virial}
There exists $B_0>100$ and $\mu_0>0$ such that, for all $B\ge B_0$,
\[
\int_{|y|<\frac B2} \left[ 3 (\partial_y\varepsilon)^2 + \varepsilon^2 - 5 Q^4 \varepsilon^2 + 20 y Q'Q^3 \varepsilon^2\right]
\geq \mu_0 \int_{|y|< \frac B2} \left(\varepsilon^2+(\partial_y\varepsilon)^2\right) - \frac 1B \int \varepsilon^2 e^{-\frac {|y|}2}.
\]
\end{lemma}
Hence, 
\begin{equation*}
 f_{1,1}^{\sim} +\mu_0 \int_{|y|<\frac{B}2}\left(\varepsilon^2+(\partial_y\varepsilon)^2\right)
 \varphi_B' \lesssim \frac1{B} \int \varepsilon^2\varphi_B' .
\end{equation*}
To estimate $f_{1,2}^{\sim}$ and $f_{1,4}^\sim$, we need the following estimate for $k=0,1$,
\begin{equation} \label{cons:BS:Wb3Wb'}
 \left|\partial_y^k (W_b^4-Q^4)\right|
 \lesssimD |s|^{-1}\omega +|s|^{-4-2k}\ONE_{[-2\delta,2\delta]}(\lambda y)+|s|^{-20}(\log|s|)^4\ONE_{[-2,0]}(|b|^{\gamma}y).
\end{equation}
Indeed, for $k=0$, we observe that
\begin{equation*}
 \left| W_b^4-Q^4\right|\lesssim \left| W^4- Q^4 \right| +\left|W_b^4- W^4 \right| 
 \lesssim Q^3|V|+V^4+|b||P_b||W|^3+|b|^4|P_b|^4.
\end{equation*}
For $k=1$, we expand 
\begin{equation*}
 W_b^3\partial_yW_b-Q^3Q'=(W^3-Q^3+3bW^2P_b+3b^2WP_b^2+b^3P_b^3)\partial_y W_b+Q^3\partial_y(V+bP_b)
\end{equation*}
which implies that
\begin{equation*}
 \left|W_b^3\partial_yW_b-Q^3Q'\right| \lesssim \left(Q^2|V|+|V|^3+|b|W^2|P_b|+|b|^3|P_b|^3 \right)|\partial_yW_b|+Q^3\left(|\partial_yV|+|b||\partial_yP_b|\right).
\end{equation*}
Thus, \eqref{cons:BS:Wb3Wb'} follows from \eqref{est:W0}, \eqref{est:V0}, \eqref{def:P:b} and \eqref{BS}.

\smallskip

We decompose 
\begin{align*}
(&W_b+\varepsilon)^6 - W_b^6 - 6 (W_b+\varepsilon)^5 \varepsilon +15Q^4\varepsilon^2 \\ &=(W_b+\varepsilon)^6 - W_b^6-6W_b^5\varepsilon-15W_b^4\varepsilon^2
-6\left[(W_b+\varepsilon)^5-W_b^5-5W_b^4\varepsilon \right] \varepsilon
-15\left[W_b^4-Q^4\right]\varepsilon^2 ,
\end{align*}
so that 
\begin{align*}
\left| (W_b+\varepsilon)^6 - W_b^6 - 6 (W_b+\varepsilon)^5 \varepsilon +15Q^4\varepsilon^2 \right| \lesssim \left(|W_b|^3|\varepsilon|+\varepsilon^4 \right)\varepsilon^2+\left|W_b^4-Q^4\right|\varepsilon^2 .
\end{align*}
Thus, it follows from \eqref{est:W0:b}, \eqref{GN:BS} and \eqref{cons:BS:Wb3Wb'} that 
\begin{equation*}
 |f_{1,2}^{\sim}| \lesssim \left( \|W_b\|_{L^{\infty}}^3\|\varepsilon\|_{L^{\infty}}+\|\varepsilon\|_{L^{\infty}}^4+\|W_b^4-Q^4\|_{L^{\infty}} \right) \frac1B\int_{|y|<\frac{B}2}\varepsilon^2\lesssim |s|^{-1}\int_{|y|<\frac{B}2} \varepsilon^2 \varphi_B' .
\end{equation*}
Similarly, \eqref{est:W0:b}, \eqref{est:W12:b} and \eqref{GN:BS} yield 
\begin{align*}
|f_{1,3}^{\sim}| &\lesssim \frac1B \int_{|y|<\frac{B}2} \left(W_b^2|\varepsilon|^3+|\varepsilon|^5 \right) |y||\partial_yW_b|\\ &
\lesssim \left( \|W_b\|_{L^{\infty}}^2\|\varepsilon\|_{L^{\infty}}+\|\varepsilon\|_{L^{\infty}}^3 \right)\|\partial_yW_b\|_{L^{\infty}} \int_{|y|<\frac{B}2}\varepsilon^2 \lesssimD |s|^{-1}B \int_{|y|<\frac{B}2} \varepsilon^2\varphi_B'.
\end{align*}
Finally, \eqref{cons:BS:Wb3Wb'} implies that
\begin{equation*}
 |f_{1,4}^{\sim}| \lesssim \left\| W_b^3\partial_yW_b-Q^3Q'\right\|_{L^{\infty}} \int_{|y|<\frac{B}2} \varepsilon^2
 \lesssimD |s|^{-1}B\int_{|y|<\frac{B}2} \varepsilon^2 \varphi_B'.
\end{equation*}
Hence, by gathering these estimates and taking $|s|$ large enough,
\begin{equation} \label{est:f1sim}
 f_{1}^{\sim} +\frac{\mu_0}2 \int_{|y|<\frac{B}2}\left(\varepsilon^2+(\partial_y\varepsilon)^2\right)\varphi_B'
 \lesssimD \frac1{B} \int \varepsilon^2\varphi_B' .
\end{equation}

The proof of \eqref{est:f1} follows by combining \eqref{est:f1<}, \eqref{est:f1>}, \eqref{est:f1sim} and by fixing $2^{-4}\min\{1,\mu_0\}$, $B\ge B_0$ large enough and taking $|s|$ large enough possibly depending on $B$.

\smallskip

\noindent \emph{Estimate for $f_2$.} We claim
\begin{equation} \label{est:f2}
|f_2| \le \frac{\mu_1}2 \int \varepsilon^2\varphi_B'+ c_\delta C_2^{\star}|s|^{-11}(\log|s|)^2 .
\end{equation}

Using the expression of $\cE_b(W)$ in \eqref{def:EbW},
\begin{align*}
f_2 &= 
2 \int \left(\vec{m} \cdot \vec{M}Q \right)G_B(\varepsilon) -2 \int \left( \Psi_M + \Psi_{W} + \Psi_b\right) G_B(\varepsilon)=: f_{2,1}+f_{2,2}.
\end{align*}
We deal first with $f_{2,1}$.
Integrating by parts and using $\cL\Lambda Q = -2Q$, $\cL Q'=0$, we compute
\begin{align*}
\int \Lambda Q \, G_B(\varepsilon)
&= - 2 \int \psi_B Q\varepsilon -\int \psi_B'(\Lambda Q)' \varepsilon +\int (\varphi_B-\psi_B)\Lambda Q \, \varepsilon
\\ & \quad -\int \Lambda Q \psi_B \left[ (W_b+\varepsilon)^5-W_b^5-5W_b^4\varepsilon\right] -5\int \Lambda Q \psi_B \left[W_b^4-Q^4\right]\varepsilon
\end{align*}
and
\begin{align*}
\int Q' \, G_B(\varepsilon)
&= -\int \psi_B'Q'' \varepsilon +\int (\varphi_B-\psi_B)Q' \varepsilon
\\ & \quad -\int Q' \psi_B \left[ (W_b+\varepsilon)^5-W_b^5-5W_b^4\varepsilon\right]-5\int Q' \psi_B \left[ W_b^4-Q^4\right]\varepsilon.
\end{align*}
Note that by the orthogonality relations \eqref{ortho}, we also have $\int yQ' \varepsilon=\int ( \Lambda Q-\frac12Q) \varepsilon=0$.
Using the definitions of $\psi_B$, $\psi_B'$ and $\varphi_B$, it follows that 
\begin{align*}
&\left|\int \psi_B Q\varepsilon \right|=\left| \int (\psi_B-1) Q \varepsilon\right| \lesssim \int_{y<-\frac{B}2} e^{-\frac{|y|}2} |\varepsilon| \lesssim e^{-\frac{B}8} \|\varepsilon\|_{L^2_\loc} , \\
&\left|\int \psi_B' Q'' \varepsilon \right|+\left|\int \psi_B'(\Lambda Q)' \varepsilon \right| \lesssim \int_{y<-\frac{B}2} e^{-\frac{|y|}2} |\varepsilon| \lesssim e^{-\frac{B}8} \|\varepsilon\|_{L^2_\loc} ,
\end{align*}
and
\begin{align*}
&\left|\int (\varphi_B-\psi_B)\Lambda Q \, \varepsilon \right|=2\left| \int(\varphi_B-\psi_B-\frac{y}B)\Lambda Q \, \varepsilon\right| \lesssim \int_{|y|>\frac{B}2} e^{-\frac{|y|}2}|\varepsilon| \lesssim e^{-\frac{B}8} \|\varepsilon\|_{L^2_\loc} , \\ 
&\left|\int (\varphi_B-\psi_B) Q' \varepsilon \right|=\left|\int (\varphi_B-\psi_B-\frac{y}B)Q' \varepsilon \right|
\lesssim \int_{|y|>\frac{B}2} e^{-\frac{|y|}2}|\varepsilon| \lesssim e^{-\frac{B}8} \|\varepsilon\|_{L^2_\loc} .
\end{align*}
Next, we get from \eqref{cons:BS:Wb3Wb'} that 
\begin{equation*}
\int \psi_B\left(|\Lambda Q|+|Q'|\right) \left|W_b^4-Q^4\right||\varepsilon|
\lesssimD |s|^{-1}\int e^{-\frac{|y|}2}|\varepsilon| \lesssimD |s|^{-1}\|\varepsilon\|_{L^2_\loc} .
\end{equation*}
Finally, we deduce from \eqref{est:W0:b} and \eqref{GN:BS}
that 
\begin{align*}
&\int \psi_B\left(|\Lambda Q|+|Q'|\right)\left|(W_b+\varepsilon)^5-W_b^5-5W_b^4\varepsilon\right| 
\\ & \lesssim \int \psi_B e^{-\frac{|y|}2} \left(|W_b|^3\varepsilon^2+|\varepsilon|^5\right)
\lesssim \left(\|W_b\|_{L^{\infty}}^3+\|\varepsilon\|_{L^{\infty}}^3\right)\|\varepsilon\|_{L^2_\loc}^2
\lesssim \|\varepsilon\|_{L^2_\loc}^2.
\end{align*}
Combining these estimates with \eqref{eL2B}, \eqref{BS2}, \eqref{cons:BS.1} and taking $B$ large enough, we obtain
\begin{equation} \label{est:f21}
|f_{2,1}| \le \frac{\mu_1}4 \int \varepsilon^2\varphi_B'+c_\delta |s|^{-12}(\log|s|)^2 .
\end{equation}
We fix $B$ to such a value (independent of $C_2^\star$ and $\delta$) and we do not track anymore this constant.

To estimate $\mathrm{f}_{2,2}$, we need an estimate on the error term $\Psi = \Psi_M + \Psi_{W} + \Psi_b$.
\begin{lemma}\label{BS:lemma}
For all $s\in \cI$, 
$\cN_B(\Psi) \lesssimD |s|^{-1} \|\varepsilon\|_{L^2_\loc} + s^{-6} \log |s|$.
\end{lemma}
\begin{proof}[Proof of Lemma~\ref{BS:lemma}]
First, we observe that \eqref{est:Psi_W} and \eqref{BS} imply 
$\cN_B(\Psi_W)\lesssimD s^{-6} \log |s|$. Second, we also have 
$\cN_B(\Psi_b)\lesssimD s^{-6} \log |s|$ by \eqref{est:Psi0:b} and \eqref{BS}.

Last, we claim that $\cN_B(\Psi_M) \lesssimD |s|^{-1} \|\varepsilon\|_{L^2_\loc} + s^{-6} \log |s|$.
By the definition of $\Psi_M$ in \eqref{def:PsiM},
\begin{align*}
\cN_B(\Psi_M)
&\lesssim \left|\vec m\right| \left( \cN_B(\Psi_{\lambda})+\cN_B(\Psi_{\sigma})+|b|\cN_B(\Lambda P_b) + |b| \cN_B(\partial_y P_b)\right)\\
&\quad + |b_s| \cN_B\left((\chi_b+\gamma y\partial_y\chi_b)P\right)
+|b| \left( \cN_B( \Psi_{\lambda})+|b|\cN_B(\Lambda P_b)\right)\\
&\quad +\cN_B\left( (\widetilde{\beta}-\beta)\left(\Lambda Q + \Psi_{\lambda} + b \Lambda P_b\right)\right).
\end{align*}
From \eqref{cons:BS.1}, \eqref{cons:BS.2} and \eqref{BS}, it follows that $|b|\lesssimD |s|^{-5}\log|s|$,
$\left|\vec m \right|\lesssimD \|\varepsilon\|_{L^2_\loc} + s^{-6} \log |s|$ and
$|b_s|\lesssimD |s|^{-1}\|\varepsilon\|_{L^2_\loc} + s^{-6} \log |s|$.
Next, from \eqref{est:Psi_lambda}-\eqref{est:Psi_sigma}, 
$\cN_B(\Psi_\lambda)+\cN_B(\Psi_\sigma)\lesssim |s|^{-1}$
and from the definition of $P$ and \eqref{Lambda:P:b.1}, $\cN_B(\Lambda P_b)+\cN_B(\partial_y P_b)+\cN_B\left((\chi_b+\gamma y\partial_y\chi_b)P\right)\lesssim 1$.

Finally, from~\eqref{def:theta3}, we have $|\widetilde\beta-\beta|+|\partial_y\beta|\lesssimD |s|^{-1} \ONE_{[\delta,\infty)}(\lambda |y|)$.
By the properties of $\varphi_B$ and~$Q$, \eqref{est:Psi_lambda} and \eqref{Lambda:P:b.1}, it follows that
$\cN_B((\widetilde\beta-\beta)(\Lambda Q + \Psi_{\lambda} + b \Lambda P_b))\lesssimD |s|^{-1}e^{-\frac {\delta}{2B\lambda}}$.
These estimates are sufficient to prove the claim on $\Psi_M$.
\end{proof}

We compute from the definition of $G_B(\varepsilon)$
\begin{equation*} 
\mathrm{f}_{2,2}=2\int \psi_B (\partial_y\Psi) (\partial_y \varepsilon)+2\int \varphi_B \Psi \, \varepsilon
-2\int \psi_B \Psi \left[(W_b+\varepsilon)^5-W_b^5 \right] .
\end{equation*}
We deduce from the Cauchy-Schwarz inequality, \eqref{cut:onR}, Lemma~\ref{BS:lemma} and then~\eqref{BS2} and~\eqref{eL2B} that
\begin{align*}
\left|\int \psi_B (\partial_y\Psi) (\partial_y\varepsilon) \right|
&\leq \cN_B(\Psi) \cN_B(\varepsilon)
\lesssimD \left(|s|^{-1} \|\varepsilon\|_{L^2_\loc}+ |s|^{-6} \log |s| \right) \cN_B(\varepsilon)\\
&\lesssimD 
|s|^{-\frac 12} \|\varepsilon\|_{L^2_\loc}^2
+|s|^{-\frac 32}[\cN_B(\varepsilon)]^2
+C_2^\star |s|^{-11} (\log |s|)^2 \\
&\lesssimD |s|^{-\frac 12} \int \varepsilon^2 \varphi_B'+C_2^\star |s|^{-11} (\log |s|)^2 .
\end{align*}
Similarly,
\begin{equation*}
\left|\int \varphi_B \Psi \varepsilon \right|
\le \cN_B(\Psi) \cN_B(\varepsilon)
\lesssimD |s|^{-\frac 12} \int \varepsilon^2 \varphi_B'+C_2^\star |s|^{-11} (\log |s|)^2.
\end{equation*}
Moreover, since $\psi_B\leq 3 \varphi_B$ and
$\|W_b\|_{L^{\infty}}+\|\varepsilon\|_{L^{\infty}}\lesssim 1$ by
\eqref{est:W0:b}, \eqref{GN:BS}, we check
\begin{equation*}
\left|\int \psi_B \Psi \left[(W_b+\varepsilon)^5-W_b^5 \right] \right|
\lesssim \int \varphi_B |\Psi|\left(|W_b|^4 |\varepsilon|+|\varepsilon|^5 \right)
\lesssim \cN_B(\Psi)\cN_B(\varepsilon).
\end{equation*}
Thus,
\begin{equation} \label{est:f22}
|f_{2,2}| \leq \frac{\mu_1}4 \int \varepsilon^2\varphi_B'+ c_\delta C_2^{\star}|s|^{-11}(\log|s|)^2 ,
\end{equation}
which together with \eqref{est:f21} yields \eqref{est:f2}. 

\smallskip 

\noindent \emph{Estimate for $f_3$.}
Integrating by parts, we compute
\begin{align*}
f_3 =\left(\frac{\sigma_s}{\lambda} - 1\right) 
\int & \biggl[ - \psi_B' (\partial_y\varepsilon)^2 -\varphi_B' \varepsilon^2 
+\frac 13 \psi_B' \left( (W_b+\varepsilon)^6- W_b^6 - 6 W_b^5 \varepsilon\right) \\
&+ 2 \psi_B (\partial_y W_b) \left( (W_b+\varepsilon)^5 - W_b^5 - 5 W_b^4 \varepsilon\right)\biggr]
\end{align*}
By \eqref{cons:BS.1}, \eqref{eL2B} and \eqref{BS2}, we have $|\vec m|\lesssimD C_2^\star |s|^{-5}\log|s|$.
From \eqref{est:W0:b}-\eqref{est:W12:b} and~\eqref{GN:BS}, $\|W_b\|_{L^\infty}+\|\partial_y W_b\|_{L^\infty}+\|\varepsilon\|_{L^\infty}
\lesssim 1$.
Thus, by \eqref{BS2} and the properties of $\varphi_B$, $\psi_B$,
\begin{equation}\label{est:f3}
 |f_3| \lesssim |\vec{m}| [\cN_b(\varepsilon)]^2 \lesssimD (C_2^{\star})^3|s|^{-15}(\log|s|)^3 \lesssim s^{-14}.
\end{equation}

\smallskip

\noindent \emph{Estimate for $f_4$.} We claim 
\begin{equation} \label{est:f4}
f_4 \leq \frac{\mu_1}8\int \left( \varepsilon^2+(\partial_y\varepsilon)^2\right) \varphi_B'+c_\delta |s|^{-14} .
\end{equation}
By integration by parts (see also~\cite[page 97]{MaMeRa1}), we have the identities
\begin{align*} 
\int (\Lambda \varepsilon) \partial_y(\psi_B\partial_y\varepsilon)
&=-\int \psi_B(\partial_y\varepsilon)^2+\frac12\int y\psi_B'(\partial_y\varepsilon)^2 , \\ 
\int (\Lambda \varepsilon) \varepsilon \varphi_B
&=-\frac12 \int y \varphi_B' \varepsilon^2\\
\int (\Lambda \varepsilon) \psi_B \left( (W_b+\varepsilon)^5-(W+F)^5\right)
&=\frac16 \int (2\psi_B-y\psi_B')\left( (W_b+\varepsilon)^6-W_b^6-6W_b^5\varepsilon \right)\\
&\quad 
-\int \psi_B \Lambda W_b\left( (W_b+\varepsilon)^5-W_b^5-5W_b^4\varepsilon \right).
\end{align*}
Hence,
\begin{align*}
f_4
&=-\frac{\lambda_s}{\lambda}\left(\int y\psi_B'(\partial_y\varepsilon)^2+2\int \varphi_B \varepsilon^2+\int y\varphi_B' \varepsilon^2\right)\\
& \quad +\frac13 \frac{\lambda_s}{\lambda}\int y\psi_B' \left( (W_b+\varepsilon)^6-W_b^6-6W_b^5\varepsilon \right)\\
& \quad +2 \frac{\lambda_s}{\lambda}\int \psi_B \Lambda W_b \left( (W_b+\varepsilon)^5-W_b^5-5W_b^4\varepsilon \right)\\
&=:f_{4,1}+f_{4,2}+f_{4,3} .
\end{align*}
First, from \eqref{cons:BS.1}, \eqref{BS}, \eqref{BS2}, and next the definition of $\widetilde \beta$ in \eqref{def:beta_tilde}
(with $c_1=-2$), we have
\[
\frac {\lambda_s}{\lambda} \geq - \widetilde \beta - C_2^\star |s|^{-5}\log|s|\geq 2 |s|^{-1} -c |s|^{-2} \geq 0,
\]
so that using $\psi_B'\geq 0$, $\varphi_B\geq 0$ and $\varphi_B'\geq 0$,
\[
-\frac{\lambda_s}{\lambda}\left(\int_{y>0} y\psi_B'(\partial_y\varepsilon)^2+2\int \varphi_B \varepsilon^2
+\int_{y>0} y\varphi_B' \varepsilon^2\right) \leq 0.
\]
Thus, by $|\frac {\lambda_s}{\lambda}|\lesssim |s|^{-1}$,
the properties of $\psi_B$, $\varphi_B$ and the H\"older and Young inequalities
\begin{align*}
f_{4,1}
&\lesssim |s|^{-1} \int_{y<0} |y| e^{\frac{y}B} \left(\varepsilon^2+(\partial_y\varepsilon)^2\right)\\
&\lesssim |s|^{-1} \left(\int_{y<0} |y|^{100} e^{\frac{y}B}\left( \varepsilon^2+(\partial_y\varepsilon)^2\right)\right)^{\frac 1{100}}
\left( \int_{y<0} \left(\varepsilon^2+(\partial_y\varepsilon)^2\right) \varphi_B'\right)^{\frac{99}{100}}\\
&\lesssimD 
|s|^{-\frac{101}2}
+|s|^{-\frac 12} \int_{y<0} \left( \varepsilon^2+(\partial_y\varepsilon)^2\right) \varphi_B' 
\end{align*}
where we have also used $\int_{y<0} |y|^{100} e^{\frac{y}B}\left( \varepsilon^2+(\partial_y\varepsilon)^2\right)\lesssim \|\varepsilon\|_{H^1}^2 \lesssimD 1$ (see \eqref{BS}).

To deal with $f_{4,2}$, we observe that $\psi_B'\equiv 0$ for $y>-B$. Moreover, note that using \eqref{est:W0:b} and \eqref{GN:BS} 
\[
\left| (W_b+\varepsilon)^6-W_b^6-6W_b^5\varepsilon \right| \lesssim |W_b|^4 \varepsilon^2 + \varepsilon^6
\lesssimD \omega \varepsilon^2 + |s|^{-4} \varepsilon^2.
\]
Then, it follows arguing as for $f_{4,1}$ that 
\begin{align*}
 |f_{4,2}| &\lesssimD |s|^{-1}\|\varepsilon\|_{L^2_{\loc}}^2
 +|s|^{-5}\int_{y<0}B^{-1}|y|e^{\frac{y}B} \varepsilon^2 \\ &
 \lesssimD |s|^{-1} \int \varepsilon^2 \varphi_B'
 +|s|^{-5}\left(\int_{y<0}|y|^{100}e^{\frac{y}B}\varepsilon^2 \right)^{\frac1{100}}\left(\int \varepsilon^2\varphi_B' \right)^{\frac{99}{100}}\\ 
 & \lesssimD |s|^{-401}+|s|^{-1} \int \varepsilon^2 \varphi_B'.
\end{align*}

To deal with $f_{4,3}$, we observe from \eqref{est:W0:b}
and \eqref{GN:BS} that
\[
\left|\Lambda W_b \left[ (W_b+\varepsilon)^5-W_b^5-5W_b^4\varepsilon \right] \right| \lesssim |\Lambda W_b| \left(|W_b|^3\varepsilon^2 +|\varepsilon|^5\right)
\lesssimD \omega \varepsilon^2 + |s|^{-4}\varepsilon^2.
\]
Hence, it follows from \eqref{cut:onR}, \eqref{eL2B} and \eqref{BS2} that
\begin{equation*}
 |f_{4,3}| \lesssimD |s|^{-1} \|\varepsilon\|_{L^2_\loc}^2+|s|^{-5}\cN_B(\varepsilon)^2
 \lesssimD |s|^{-1} \int \varepsilon^2 \varphi_B+C_2^{\star}|s|^{-15}\log|s|.
\end{equation*}
We conclude the proof of \eqref{est:f4} by combining the estimates for $f_{4,1}$, $f_{4,2}$ and $f_{4,3}$. 

\smallskip
\noindent \emph{Estimate for $f_5$.} We claim that for $|s|$ large enough possibly depending on $B$ and $C_2^{\star}$, 
\begin{equation} \label{est:f5}
|f_5| \leq \frac{\mu_1}{8}\int \varepsilon^2 \varphi_B'+c_\delta |s|^{-14}. 
\end{equation}
We claim
\begin{equation}\label{on:Wb}
|\partial_s W_b|\lesssim |s|^{-2}.
\end{equation}
Indeed, by the definitions of $W_b$ and $W$, and then \eqref{eq:dsthj}, we have
\begin{align*}
\partial_s W_b & = \partial_s W + \partial_s (bP_b) \\
&= \sum_{j=1}^5V_j \partial_s(\theta^j)
+(\log\lambda)\sum_{k=4}^5 V_k^{\star}\partial_s(\theta^j)
+\frac{\lambda_s}{\lambda} \sum_{k=4}^5 V_k^{\star} \theta^j
+ b_s \left( \chi_b + \gamma y \partial_y \chi_b\right) P\\
&= \frac{\lambda_s}{\lambda} \left(\sum_{j=1}^5 jV_j\theta^{j-1} \Lambda \theta
+(\log\lambda)\sum_{j=4}^5 j V_j^{\star} \theta^{j-1} \Lambda \theta
+\sum_{k=4}^5 V_k^{\star} \theta^j \right)\\
&\quad +\frac{\sigma_s}{\lambda} \left(\sum_{j=1}^5 j V_j \theta^{j-1} \partial_y \theta
+(\log\lambda)\sum_{j=4}^5 j V_j^{\star} \theta^{j-1} \partial_y \theta\right)\\
&\quad +b_s\left( \chi_b + \gamma y \partial_y \chi_b\right) P.
\end{align*}
Using \eqref{on:theta} and \eqref{supp:dtheta}, we have
$|\theta|+|\Lambda\theta| \lesssim |s|^{-1}$ and $|\partial_y \theta| \lesssim |s|^{-3}$.
By \eqref{VVV}, this yields
\[
\left|\sum_{j=1}^5 jV_j\theta^{j-1} \Lambda \theta\right|
+\left|(\log\lambda)\sum_{j=4}^5 j V_j^{\star} \theta^{j-1} \Lambda \theta\right|
+\left|\sum_{k=4}^5 V_k^{\star} \theta^j \right|\lesssim |s|^{-1}
\]
and
\[
\left| \sum_{j=1}^5 j V_j \theta^{j-1} \partial_y \theta\right|
+\left|(\log\lambda)\sum_{j=4}^5 j V_j^{\star} \theta^{j-1} \partial_y \theta\right|\lesssim |s|^{-2}.
\]
Using \eqref{cons:BS.1}, \eqref{cons:BS.2}, \eqref{BS} and \eqref{BS2}, we have
$| \frac{\lambda_s}{\lambda}| \lesssim |s|^{-1}$, $ \left| \frac{\sigma_s}{\lambda}\right| \lesssim 1$
and $|b_s|\lesssimD C_2^\star |s|^{-6}\log|s|$.
Thus, \eqref{on:Wb} is proved.

Using \eqref{est:W0:b}, we have
$|W_b|^3 \lesssim \omega + |s|^{-3}$. Moreover, 
$
\left| (W_b+\varepsilon)^5 - W_b^5 - 5 W_b^4 \varepsilon \right|
\lesssim |W_b|^3 \varepsilon^2 + |\varepsilon|^5$, and it 
 follows from \eqref{cut:onR} and \eqref{GN:BS} that
\[
|f_5| \lesssim |s|^{-2} \|\varepsilon\|_{L^2_\loc}^2 + |s|^{-5} [\cN_B(\varepsilon)]^2 ,
\]
which yields \eqref{est:f5}. 

Finally, we conclude the proof of \eqref{time:local} by gathering \eqref{est:f1}, \eqref{est:f2}, \eqref{est:f3}, \eqref{est:f4} and \eqref{est:f5}. 

\smallskip

We turn to the proof of \eqref{coer:local}. We decompose $\cF$ as follows: 
\begin{align*}
 \lambda^2\cF &=\int \left[(\partial_y\varepsilon)^2 \psi_B + \varepsilon^2 \varphi_B-5Q^4\varepsilon^2 \right]
- \frac13 \int \left[ (W_b+\varepsilon)^6 - W_b^6 - 6 W_b ^5 \varepsilon-15Q^4\varepsilon^2\right] \psi_B \\ 
&=: \cF_1+\cF_2 .
\end{align*}
To bound $\cF_1$ from below, we rely on the coercivity of the linearized operator around the ground state $\cL$ under the orthogonality conditions \eqref{ortho} (see \eqref{coercivity.2}) and standard localisation arguments. Proceeding for instance as \cite[Appendix A]{MaMe} or as in \cite[Proof of Lemma 3.5]{CoMa1}, we deduce that there exists $\tilde{\mu}_2>0$ such that, for $B$ large enough, 
$\cF_1 \ge \tilde{\mu}_2 [\cN_B(\varepsilon)]^2 $.

To estimate $\cF_2$, we write 
\begin{equation*}
 (W_b+\varepsilon)^6 - W_b^6 - 6 W_b ^5 \varepsilon-15Q^4\varepsilon^2
 =(W_b+\varepsilon)^6 - W_b^6 - 6 W_b ^5 \varepsilon-15W_b^4\varepsilon^2-15\left( W_b^4-Q^4\right)\varepsilon^2 ,
\end{equation*}
so that 
\begin{equation*}
 \left|(W_b+\varepsilon)^6 - W_b^6 - 6 W_b ^5 \varepsilon-15Q^4\varepsilon^2 \right|
 \lesssim \left(|W_b|^3|\varepsilon|+\varepsilon^4\right)\varepsilon^2+\left| W_b^4-Q^4\right|\varepsilon^2.
\end{equation*}
Thus, it follows from \eqref{cut:onR}, and then \eqref{est:W0:b}, \eqref{cons:BS:Wb3Wb'} and \eqref{GN:BS} that
\begin{align*}
 \left|\cF_2 \right| \lesssim \left( \|W_b\|_{L^{\infty}}^3\|\varepsilon\|_{L^{\infty}}+\|\varepsilon\|_{L^{\infty}}^4+\left\| W_b^4-Q^4\right\|_{L^{\infty}}\right) \int \varepsilon^2\varphi_B 
 \lesssimD |s|^{-1}[\cN_B(\varepsilon)]^2 .
\end{align*}
The proof of \eqref{coer:local} follows from these estimates taking $|s|$ large enough
\end{proof}

\subsection{Closing the energy estimates}
\begin{lemma} \label{lemma:eps_loc:BS}
There exists $C_2^{\star}>1$ such that on $\cI$,
\begin{equation}\label{BS2:improved}
\cN_B(\varepsilon(s)) 
+\left(\int_{S_n}^s \left(\frac\tau s\right)^4 \|\varepsilon(\tau)\|_{L^2_\loc}^2\, d\tau \right)^{\frac 12} 
\leq \frac{C_2^{\star}}2 |s|^{-5} \log |s|. 
\end{equation}
\end{lemma} 

\begin{proof}

Let $s\in\cI$.
Integrating~\eqref{time:local} on $[S_n,s]$ and using $\cF(S_n)=0$, we find
\[
\cF(s) + \int_{S_n}^s \tau^4 \int \left( (\partial_y \varepsilon)^2 + \varepsilon^2 \right) \varphi_B' d\tau
\lesssimD C_2^\star s^{-6} (\log|s|)^2 .
\]
Thus, by \eqref{eL2B} and~\eqref{coer:local}, we obtain
\[
[\cN_B(\varepsilon(s))]^2 
+s^{-4} \int_{S_n}^s \tau^4 \|\varepsilon(\tau)\|_{L^2_\loc}^2 d\tau 
\lesssimD C_2^\star s^{-10} (\log |s|)^2.
\]
By choosing $C_2^\star>1$ large enough, this implies \eqref{BS2:improved}.
\end{proof}

We complete the proof of Proposition~\ref{prop:bootstrap}. The constant $C_1^{\star}>1$ has been fixed as in Lemma~\ref{lemma:param:BS}, and the constant $C_2^{\star}>1$ is now chosen as in Lemma~\ref{lemma:eps_loc:BS}.
Assuming by contradiction that $S_n^{\star}<s_0$, by continuity, it follows that at least one of the estimates in \eqref{eq:H1}, \eqref{BS} and \eqref{BS2} is reached, which is absurd by Lemmas~\ref{lemma:param:BS}, \ref{le:mass-ener} and \ref{lemma:eps_loc:BS}.

\section{Proof of the main result}\label{S:6}

We follow the presentation in \cite[Section 3.8]{CoMa1} for the compactness argument.
By Proposition~\ref{prop:bootstrap}, the sequence $\{u_n\}$ of solutions of \eqref{gkdv}
defined in Section~\ref{bootstrap:sec} satisfy the uniform estimates~\eqref{BS}-\eqref{BS2} on $[S_n,s_0]$ where $s_0<-1$ is independent of $n$. We rewrite these estimates in the time variable $t$.
From \eqref{BS}, $\lambda_n^3(s)=|s|^{-6}+\cO(|s|^{-7} \log|s|)$ so that \eqref{defts} implies
\[
t-T_n=\int_{S_n}^s \lambda_n^3(s') ds'
=\frac 15\left( \frac1{|s|^5}-\frac1{|S_n|^5}\right)+\cO\left(\frac{\log|s|}{|s|^6}\right)
\]
where $T_n=\frac 1{5|S_n|}$. It follows that
$
t=\frac1{5|s|^5} + \cO\left( \frac{\log|s|}{|s|^6}\right)$, and so
\begin{equation}\label{time}
\frac 1{|s|^5} = 5 t +\cO\left( t^\frac 65 |\log t|\right).
\end{equation}
Thus, estimates~\eqref{BS} and \eqref{GN:BS} on $[S_n,s_0]$ imply the existence of $t_0>0$ independent of $n$ such that on $[T_n,t_0]$, it holds
\begin{equation}\label{BS:t}\begin{aligned}
|b_n(t)|&\lesssimD t |\log t|,\\
\left| \lambda_n(t) - (5t)^{\frac 25}\right| &\lesssimD t^{\frac 35}|\log t|,\\
\left| \sigma_n(t) - (5t)^{\frac 15} \right| &\lesssimD t^{\frac 25}|\log t|,\\
\|\varepsilon_n(t)\|_{L^2} \lesssim \delta^\frac 12,\quad &
\|\partial_y \varepsilon_n(t)\|_{L^2} \lesssim \delta^{-\frac 12} t^{\frac 25}.
\end{aligned}
\end{equation}These estimates imply that the sequence $\{u_n(t_0)\}$ is bounded in $H^1$ and that 
 the sequences $\{\lambda_n(t_0)\}$, $\{\sigma_n(t_0)\}$ and $\{1/\lambda_n(t_0)\}$ are bounded in $\RR$.
Let
\[
\tilde u_n (x) = \lambda_n^\frac 12(t_0) u_n (t_0,\lambda_n(t_0) x + \sigma_n(t_0))
=W_{b,n}(t_0,x) + \varepsilon_n(t_0,x) .
\]
Then $\{\tilde u_n\}$ is also bounded in $H^1$ (see \eqref{est:W0:b} and \eqref{est:W12:b}).
Therefore, there exist subsequences $\{\tilde u_{n_k}\}$, $\{\lambda_{n_k}\}$, $\{\sigma_{n_k}\}$, $\{b_{n_k}\}$ of $\{\tilde u_n\}$, $\{\lambda_n\}$, $\{\sigma_n\}$, $\{b_n\}$ and $\tilde u_0\in H^1$, $\lambda_\infty>0$, $\sigma_\infty\in \RR$, $b_\infty\in\RR$
such that $\tilde u_{n_k} \rightharpoonup \tilde u_0$ weakly in $H^1$
and $\lambda_{n_k} \to \lambda_\infty$, $\sigma_{n_k} \to \sigma_\infty$, $b_{n_k} \to b_\infty$ as $k\to \infty$.

For any $k\geq 0$, let $\tilde u_k(t)$ be the maximal solution of \eqref{gkdv} such that $\tilde u_k(0)=\tilde u_{n_k}$, and
let $(\tilde\lambda_k,\tilde\sigma_k,\tilde b_k,\tilde\varepsilon_k)$ denote its decomposition as given by Section~\ref{S:3.2}.
By the scaling invariance
\begin{equation*}
 \tilde u_k(t,x)=\lambda_{n_k}^{\frac12}(t_0) u_{n_k}(t_0+\lambda_{n_k}^3(t_0)t,\lambda_{n_k}(t_0)x+\sigma_{n_k}(t_0)), \quad \forall \, t \in \left[-\frac{t_0-T_{n_k}}{\lambda_{n_k}^3(t_0)},0\right] .
\end{equation*}
Hence, we deduce from the uniqueness of the decomposition that
\begin{gather*}
\tilde\lambda_k(t) = \frac{\lambda_{n_k}(t_0+\lambda_{n_k}^3(t_0) t)}{\lambda_{n_k}(t_0)},\quad
\tilde\sigma_k(t)=\frac{\sigma_{n_k}(t_0+\lambda_{n_k}^3(t_0) t) - \sigma_{n_k}(t_0)}{\lambda_{n_k}(t_0)},\\
\tilde b_k(t)=b_{n_k}(t_0+\lambda_{n_k}^3(t_0) t),\quad \tilde\varepsilon_k(t)=\varepsilon_{n_k}(t_0+\lambda_{n_k}^3(t_0)t).
\end{gather*}
Let $t_1\in (0,\frac{t_0}{\lambda_\infty^3})$ and apply \cite[Lemma 2.10]{CoMa1} on $[-t_1,0]$.
It follows that the solution~$\tilde u$ of~\eqref{gkdv} with initial data $\tilde u_0$ exists on the time interval $(-\frac{t_0}{\lambda_\infty^3},0]$,
and its decomposition $(\tilde \lambda,\tilde \sigma,\tilde b,\tilde \varepsilon)$ is the limit of $\{(\tilde\lambda_k,\tilde\sigma_k,\tilde b_k,\tilde\varepsilon_k)\}$
 as $k\to \infty$ (in the weak $H^1$ sense for $\{\tilde\varepsilon_k\}$).
 
Now, let $T=5^3 t_0$ and define the solution $u$ of \eqref{gkdv} on $[0,T)$ by
\[
u(t,x) = \frac1{(5\lambda_\infty)^{\frac 12}} \tilde u \left( \frac{-t}{(5\lambda_\infty)^3},\frac{-x-\sigma_\infty}{5\lambda_\infty}\right).
\]
Denote by $(\lambda,\sigma, b,\varepsilon)$ the parameters of its decomposition on $[0,T)$,
\[
u(t,x) = \frac 1{\lambda^\frac 12(t)} \left( W_{b} \left(t, \frac{x-\sigma(t)}{\lambda(t)}\right)
+\varepsilon\left(t, \frac{x-\sigma(t)}{\lambda(t)}\right)\right)
\]
where, from~\eqref{BS:t}, for any $t\in [0,T)$,
\begin{equation}\label{BS:tilde}
\begin{aligned}
|b(t)|&\lesssimD (T-t) \left|\log (T-t)\right|,\\
\left|\lambda(t) - (T-t)^{\frac 25}\right| &\lesssimD (T-t)^{\frac 35}\left|\log (T-t)\right|,\\
\left|\sigma(t) + 5 (T-t)^{\frac 15} \right| &\lesssimD (T-t)^{\frac 25}\left|\log (T-t)\right|,\\
\|\varepsilon(t)\|_{L^2} \lesssim \delta^\frac 12,\quad&
\|\partial_y \varepsilon(t)\|_{L^2} \lesssim \delta^{-\frac 12} (T-t)^{\frac 25}.
\end{aligned}
\end{equation}
In particular, $u$ is an $H^1$ solution of \eqref{gkdv} blowing up at time $t=T$ at the point $0$.
To complete the proof of Theorem~\ref{th:1}, we write
\[
u(t,x) - \frac 1{(T-t)^{\frac 15}} Q\left( \frac {x-\sigma(t)}{(T-t)^{\frac 25}}\right) = r_1(t,x)+r_2(t,x)
\]
where
\begin{align*}
r_1(t,x) 
&= \frac 1{\lambda^\frac 12(t)} Q\left( \frac {x-\sigma(t)}{\lambda(t)}\right)
- \frac 1{(T-t)^{\frac 15}} Q\left( \frac {x-\sigma(t)}{(T-t)^{\frac 25}}\right)\\
r_2(t,x)&= \frac 1{\lambda^\frac 12(t)} v\left( \frac {x-\sigma(t)}{\lambda(t)}\right),
\end{align*}
and
\[
v(t,y) = W_{b}(t,y) - Q(y) + \varepsilon(t,y)
=V(t,y) + b(t) P_{b}(t,y) + \varepsilon(t,y),
\]
and we prove estimates on $r_1$, $r_2$.
We find setting
$\bar \lambda(t) = {(T-t)^\frac 25}/\lambda(t)$ and using \eqref{BS:tilde},
\begin{align*}
\|r_1(t)\|_{L^2} & =\left\| \bar \lambda^\frac 12(t) Q (\bar\lambda(t)\, \cdot \, ) - Q\right\|_{L^2}
\lesssim |\bar \lambda(t) -1| \lesssim (T-t)^{\frac 15} |\log (T-t)|,\\
(T-t)^{\frac 25} \|\partial_x r_1(t)\|_{L^2} & = \left\| \bar \lambda^\frac 32(t) Q' (\bar\lambda(t) \,\cdot \, ) - Q'\right\|_{L^2}
\lesssim |\bar \lambda(t) -1| \lesssim (T-t)^{\frac 15} |\log (T-t)|.
\end{align*}
Next, by \eqref{est:V0}, \eqref{def:P:b} and \eqref{BS:tilde},
\[
\|v\|_{L^2} \leq \|V\|_{L^2}+\|b P_{b}\|_{L^2}+\|\varepsilon\|_{L^2}
\lesssim \delta^\frac 12+|b|^{1-\frac\gamma2}\lesssim \delta^\frac 12.
\]
Second, by \eqref{est:V0}, \eqref{def:P:b} and \eqref{BS:tilde},
\[
\|\partial_y v\|_{L^2} \leq \|\partial_y V\|_{L^2}+\|b \partial_y P_{b}\|_{L^2}+\|\partial_y \varepsilon\|_{L^2}
\lesssim \lambda^{\frac 12} +|b|+ \delta^{-\frac 12} (T-t)^{\frac 25} \lesssimD (T-t)^{\frac 15}.
\]
Thus,
\[
\|r_2(t)\|_{L^2}=\|v(t)\|_{L^2}\lesssim \delta^\frac 12,\quad
(T-t)^\frac 25 \|\partial_x r_2(t)\|_{L^2} = \bar \lambda(t) \|\partial_y v(t)\|_{L^2}\lesssimD (T-t)^{\frac 15}.
\]

\section*{Acknowledgments} 
D.P. was supported by a Trond Mohn Foundation grant.
Part of this work was done while Y.M. was visiting the Department of Mathematics of the University of Bergen, whose hospitality is acknowledged.
The authors would like to thank Frank Merle for stimulating discussions and encouragement.

\end{document}